\documentclass[a4paper,11pt]{article}

\usepackage{graphicx}
\usepackage{cancel,enumerate}
\usepackage[version=3]{mhchem} %% to write chemical reactions
\usepackage{amsfonts, amsmath,amssymb,color,hyperref, amsthm}
\usepackage[all]{xy}
\usepackage{array}

\usepackage{tikz}
\usetikzlibrary{snakes}
\usetikzlibrary{shapes}
\input xy
\xyoption{all}

\newcommand{\mmS}{\mathcal{S}}
\newcommand{\mmC}{\mathcal{C}}
\newcommand{\mmN}{\mathcal{N}}
\newcommand{\mmE}{\mathcal{E}}
\newcommand{\mmR}{\mathcal{R}}
\newcommand{\mmU}{\mathcal{U}}
\newcommand{\mmH}{\mathcal{H}}
\newcommand{\mmRu}{\mathcal{R}_{\mathcal{U}}}

\newcommand{\mmG}{\mathcal{G}}
\newcommand{\wk}{\widetilde{\kappa}}
\renewcommand{\wr}{\widetilde{r}}
\newcommand{\wR}{\widetilde{\mmR}}
\renewcommand{\k}{\kappa}
\newcommand{\extDelta}{\Delta'}

\newcommand{\R}{{\mathbb R}}
\newcommand{\st}{\,\mid \,}
\DeclareMathOperator{\im}{im} 
\DeclareMathOperator{\supp}{supp}

\theoremstyle{plain}
\newtheorem{theorem}{Theorem}
\newtheorem{proposition}[theorem]{Proposition}

\newtheorem{lemma}{Lemma}

\theoremstyle{definition}
\newtheorem{definition}[theorem]{Definition}
\newtheorem{remark}[theorem]{Remark}
\newtheorem{example}[theorem]{Example}

\begin{document}

\title{Graphical reduction of reaction networks by linear elimination of species}

\author{Meritxell S\'aez\and Carsten Wiuf \and Elisenda Feliu}

\maketitle

\begin{abstract} 
The quasi-steady state approximation and time-scale separation are commonly applied methods to simplify models of biochemical reaction networks based on ordinary differential equations (ODEs). The concentrations of the ``fast'' species are assumed effectively to be at steady state with respect to the ``slow'' species. Under this assumption the steady state equations can be used to eliminate the ``fast'' variables and a new ODE system with only the slow species can be obtained.

We interpret a reduced system obtained by time-scale separation as the ODE system arising from a unique reaction network, by identification of a set of reactions and the corresponding rate functions. The procedure is graphically based and can easily be worked out by hand for small networks. For larger networks, we provide a pseudo-algorithm.  We study properties of the reduced network, its kinetics and conservation laws, and show that the kinetics of the reduced network fulfil realistic assumptions, provided the original network does. We illustrate our results using biological examples such as substrate mechanisms, post-translational  modification systems and networks with intermediates (transient) steps.


\end{abstract}

\medskip

\section{Introduction}
\label{intro}
Biochemical reaction networks  often involve many biochemical species that interact through many reactions. The mathematical models that are used to describe such networks can be quite complex and analytically intractable, both in terms of variables (species concentrations) as well as (unknown) parameters such as reaction rate constants. It is therefore commonplace to reduce the dimensions of the models by various means.

One way to  simplify a model  is by  time-scale separation \cite{pantea-QSSA,gunawardena-linear}. 
To apply time-scale separation, the species in the system are divided into fast and slow species. Fast species are assumed  to be at  equilibrium (a so-called quasi-equilibrium) even though the system as a whole has not reached a state of equilibrium. For example, if degradation of a species proceeds at  high rate, then it is short lived  and  becomes a  `fast' species. It is thus reasonable to assume that the species is in a state of (quasi-)equilibrium. Being at equilibrium, the fast species might be eliminated from the system, resulting in a simpler, reduced system with only the slow species.  
Tikhonov's theorem \cite{Goeke2015,tikhonov} might be used to conclude that the original and the reduced system have similar dynamics (over a compact time interval).

Our aim is to  interpret the reduced system  as an ODE system arising from a  reaction network with only the slow species. 
We follow some ideas in \cite{Fel_elim,feliu:intermediates}.
The starting point is  a system of ordinary differential equations (ODEs)  arising from a reaction network with a given kinetics.  We  assume that a set of species (`the fast species') is given and that these species  effectively are at steady state.  If the  species do not interact with each other  (i.e.\,\,they are not on the same side of a reaction), then  a  relationship between the concentrations of the fast and the slow species   can be derived under certain conditions \cite{Fel_elim}. Using this relationship, a reduced ODE system for the slow species in which the fast species have been eliminated is obtained. 

The next step is to  interpret the  ODE system for the slow species as the ODE system associated with a \emph{reduced reaction network}. This step could be carried out in various ways. We identify sequences of reactions in the original network, by means of a specific \emph{graph} (see example below), such that the net production of the fast species is zero. By contracting  the sequence into a single reaction and removing the fast species, a reduced reaction is obtained. A non-trivial issue is to determine the reaction rate of the reduced reaction. For small networks, the procedure can easily be carried out by hand.

Subsequently, we  establish that the reduced reaction network has some desirable basic properties. 
For example, if the reactions of the original reaction network cannot take place in the absence of their reactant species, then this  also holds for the reduced reaction network.  

Graphical means to re-interpret  a reaction network under quasi-stationarity are not new.
The King-Altman procedure is a systematic way to  eliminate enzymes and intermediate species, in an enzyme catalysed system with mass-action kinetics \cite{king-altman}. Later, Wong and Hanes \cite{WH62} gave a systematic method to find the  production rates of the slow species, avoiding much of the algebraic manipulations in \cite{king-altman}; see also \cite{gunawardena-linear}.
These approaches can be seen as instances of a general linear (graphical) elimination procedure \cite{Fel_elim,feliu:intermediates}.

None of the mentioned approaches specifically identifies a reduced reaction network, except for \cite{feliu:intermediates}, where a reduced reaction network is obtained after  elimination of intermediate species. A related approach by Horiuti,  Temkin and co-workers, is based on  a graphical procedure  to identify, not one but many, reduced reaction networks \cite{Rad,Temkin1}. The procedure is similar in spirit to our method and we will contrast the two approaches in Subsection \ref{sec:temkin}. 
In \cite{Gabor2015} an algorithmic method to obtain possible reaction networks corresponding to a rational ODE system is described. In \cite{Rao2014} a graphical method for the elimination of complexes (not species) is presented.

To  illustrate the results and methods of the paper, consider a \emph{ping-pong bi-bi} mechanism  \cite{FH07}, described by the reaction network
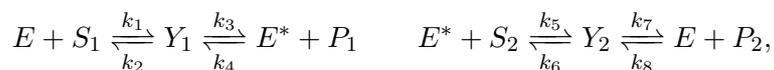
\[
\begin{array}{c}
E+S_1\ce{<=>[k_1][k_2]} Y_1 \ce{<=>[k_3][k_4]} E^*+P_1 \qquad 
E^*+S_2\ce{<=>[k_5][k_6]} Y_2 \ce{<=>[k_7][k_8]} E+P_2,
\end{array}
\]
where $E$, $E^*$ are two forms of an enzyme, $S_1$, $S_2$ are substrates, $Y_1$, $Y_2$ intermediates and $P_1$, $P_2$ products. 
The molar concentrations of the species are denoted as: $x_1=[E]$, $x_2=[E^*]$, $x_3=[S_1]$, $x_4=[S_2]$, $x_5=[P_1]$, $x_6=[P_2]$,  $x_7=[Y_1]$ and $x_8=[Y_2]$.
With this notation and assuming mass-action kinetics, the evolution of the species concentrations is described by the 
 ODE system  (cf. Equation \eqref{eq:ODE}):
\begin{align*}
\dot{x}_1&=k_{{2}}x_{{7}}+k_{{7}}x_{{8}}-(k_{{1}}x_{{3}}+k_{{8}}x_{{6}})x_1 & 
\dot{x}_2&=k_{{3}}x_{{7}}+k_{{6}}x_{{8}} -(k_{{4}}x_{{5}}+k_{{5}}x_{{4}})x_2 \\
\dot{x}_3&=k_{{2}}x_{{7}} - k_{{1}}x_{{1}}x_{{3}} & 
\dot{x}_4&= k_{{6}}x_{{8}} -k_{{5}}x_{{2}}x_{{4}} \\
\dot{x}_5& = k_{{3}}x_{{7}}-k_{{4}}x_{{2}}x_{{5}} 
&\dot{x}_6 &= k_{{7}}x_{{8}} -k_{{8}}x_{{1}}x_{{6}} \\
\dot{x}_7&= k_{{1}}x_{{1}}x_{{3}}+k_{{4}}x_{{2}}x_{{5}}-(k_{{2}}+k_{{3}})x_{{7}}&\dot{x}_8 & = k_{{5}}x_{{2}}x_{{4}}+k_{{8}}x_{{1}}x_{{6}}-(k_{{6}}+k_{{7}})x_{{8}}.
\end{align*}

The main interest is on the conversion of substrates into products \cite{enz-kinetics}. It is thus sensible to eliminate the set of species $\mmU=\{E, E^*, Y_1,Y_2\}$ 
to obtain a reduced reaction network on the species $S_1,S_2,P_1,P_2$ alone. 
Assuming that  $E, E^*, Y_1,Y_2$ are at steady state and using the conservation law for the total amount of enzyme, $T=x_1+x_2+x_7+x_8$, we obtain
\begin{align} \label{eq:ex-intro}
x_{{1}} & = q(x)((k_{{6}}+k_{{7}})k_{{2}}k_{{4}}x_5+(k_{{2}}+k_{{3}})k_{{5}}k_{{7}}x_4) \nonumber
\\
x_{{2}} & =q(x)((k_{{6}}+k_{{7}})k_{{1}}k_{{3}}x_3+(k_{{2}}+k_{{3}})k_{{6}}k_{{8}}x_6) \\
x_{{7}} &=q(x)(k_{{1}}k_{{5}}k_{{7}}x_3x_4+(k_{{6}}+k_{{7}})k_{{1}}k_{{4}}x_3x_5+k_{{4}}k_{{6}}k_{{8}}x_{{5}}x_{{6}}) \nonumber \\
x_{{8}} &=q(x)(k_{{1}}k_{{3}}k_{{5}}x_{{4}}x_{{3}}+k_{{2}}k_{{4}}k_{{8}}x_{{5}}x_{{6}}+(k_{{2}}+k_{{3}})k_{{5}}k_{{8}}x_4x_6), \nonumber
\end{align}
where 
\begin{align*}
q(x)&= T\Big{/}\Big{(}  
k_{1}k_{5}( k_{3}+k_{7} ) x_{3}x_{4}+ ( k_{6}+k_{7} ) ( k_{1}k_{4}x_{3}x_{5}   + k_{1}k_{3} x_{3}+ k_{2}k_{4} x_{5} )   
  \\
&+ k_{4}k_{8}( k_{2}+k_{6} )x_{5}x_{6}+  
( k_{2}+k_{3} )(k_{5}k_{8}x_{4}x_{6} + k_{5}k_{7} x_{4}+  k_{6}k_{8} x_{6}) \Big).
\end{align*}
Conditions that guarantee existence and positivity of functions  expressing the concentrations of the species in a set $\mmU$ in terms of the remaining species, were given in \cite{Fel_elim}.
These are reviewed in Section~\ref{elimination}. The form of the expressions is the content of  Theorem \ref{elim}. 

After substitution of \eqref{eq:ex-intro} into the ODE system, we obtain the production rates of   $P_1$, $P_2$ and the consumption rates of $S_1$, $S_2$:
\begin{equation}\label{reduced-intro}
\dot{x}_5=\dot{x}_6=-\dot{x}_3=-\dot{x}_4 =q(x)(-k_{{2}}k_{{4}}k_{{6}}k_{{8}}x_{{5}}x_{{6}}+k_{{1}}k_{{3}}k_{{5}}k_{{7}}x_{{4}}x_{{3}}).
\end{equation}
This ODE system might be interpreted as arising from a reaction network with one reversible reaction with the following rate functions:
\begin{equation}\label{ex:introred}
S_1+S_2\ce{<=>[\kappa_1][\kappa_2]} P_1+P_2,\quad \kappa_1=q(x)k_{{1}}k_{{3}}k_{{5}}k_{{7}}x_{{3}}x_{{4}}, \quad\kappa_2 = q(x)k_{{2}}k_{{4}}k_{{6}}k_{{8}}x_{{5}}x_{{6}}.
\end{equation}
In this particular case, it seems straightforward to identify  reasonable reactions that explain \eqref{reduced-intro}. However, in general this might not be so.  
A procedure to systematically find a reduced reaction network is presented in Section~\ref{reduction}. At the core of the procedure  is a graph that relates the species in $\mmU$ (Definition~\ref{defgraph}).  
The graph of the example  is 
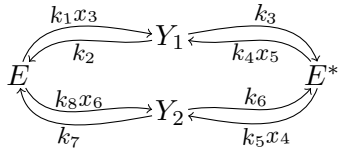
\begin{center}
\begin{tikzpicture}[inner sep=1pt]
\clip(-0.5,-1.5) rectangle (4.5,0.5);
\node (E1) at (0,-0.5) {$E$};
\node (E2) at (4,-0.5) {$E^*$};
\node (Y1) at (2,0) {$Y_1$};
\node (Y2) at (2,-1) {$Y_2$};
\draw[->] (E1) to[out=70,in=170] node[above,sloped]{\footnotesize $k_1x_3$}(Y1);
\draw[->] (Y1) to[out=190,in=50] node[below,sloped]{\footnotesize $k_2$}(E1);
\draw[->] (E2) to[out=250,in=-10] node[below,sloped]{\footnotesize $k_5x_4$}(Y2);
\draw[->] (Y2) to[out=10,in=230] node[above,sloped]{\footnotesize $k_6$}(E2);
\draw[->] (Y2) to[out=190,in=280] node[below,sloped]{\footnotesize $k_7$}(E1);
\draw[->] (E1) to[out=300,in=170] node[above,sloped]{\footnotesize $k_8x_6$}(Y2);
\draw[->] (Y1) to[out=10,in=110] node[above,sloped]{\footnotesize $k_3$}(E2);
\draw[->] (E2) to[out=130,in=-10] node[below,sloped]{\footnotesize $k_4x_5$}(Y1);
\end{tikzpicture}
\end{center}
Each edge in the graph corresponds to a reaction in the original network. For example, the edge $E\ce{->[k_8x_6]} Y_2$ corresponds to the reaction $E+P_2\ce{->[k_8]} Y_2$.

Our main result is Theorem \ref{rednet} that establishes the reduced reaction network.  Essentially, it has two types of reactions: 
 The reactions of the original network that do not involve species in $\mmU$, and reactions found by considering certain  cycles of the graph. 
 The reactant (resp.\ product) of such a reduced reaction is the sum of the reactants (resp.\ products) of the original reactions defining the cycle, after removing the species in $\mmU$.
For example,  the above graph has two cycles that give rise to two reduced reactions. The clockwise cycle involves the reactions corresponding to the labels $k_1x_3$, $k_3$, $k_5x_4$ and $k_7$:
$$
\begin{array}{c}
E+S_1\ce{->[k_1]} Y_1 \qquad Y_1 \ce{->[k_3]} E^*+P_1 \qquad  E^*+S_2\ce{->[k_5]} Y_2 \qquad Y_2 \ce{->[k_7]} E+P_2.
\end{array}
$$
Adding the reactants and products together and removing $E$, $E^*$, $Y_1$, and $Y_2$, we obtain the reaction $S_1+S_2\rightarrow P_1+P_2$. Similarly, the anti-clockwise cycle gives the reaction $P_1+P_2\rightarrow S_1+S_2$ (consistent with \eqref{ex:introred}).
The rate function of each reaction is determined from the edge labels of the graph.
 We describe an  algorithm   to find the reduced reaction network and the rate functions after Theorem \ref{rednet}.

Note that the reduced rate functions in \eqref{ex:introred} have `mass-action form' in the sense that the reactions can only occur in the presence of the reactant species. This holds generally for the reduced reaction network (cf. Section~\ref{BasicProp}).

The outline of the paper is as follows. In Section \ref{preliminaries} we introduce background material. In Section \ref{elimination} we recall some results concerning elimination of variables in \cite{Fel_elim}. The reduced reaction network is derived in Section \ref{reduction}. Properties of the reduced reaction network  in relation to the kinetics (including mass-action) and conservation laws are presented
in Section \ref{BasicProp}. This concludes the core of the paper.
Sections \ref{severalsteps}, \ref{sec:temkin2} and \ref{examples} discuss respectively iterative elimination, comparison with previous results (the approach in \cite{Temkin1} and networks with intermediates) and post-translational modification networks.
Finally, Section \ref{sec:proofs} contains  proofs.

%%%%%%%%%% Preliminaries

\section{Preliminaries}\label{preliminaries}

In this section we introduce necessary concepts from graph theory and reaction network theory.

We let $\R_{\geq 0}$ and $\R_{>0}$ denote the sets of nonnegative and positive real numbers respectively, and define $\R^n_{\geq 0}$ and $\R^n_{>0}$ accordingly.
For $x,y\in \R^n$, $x\cdot y$ denotes the scalar product associated with the Euclidean norm. Further,  $\langle v_1,\dots,v_r\rangle$ denotes the vector subspace generated by  $v_1,\dots,v_r\in \R^n$.

\paragraph{\bf Graphs, multidigraphs and spanning trees. }\label{subsec:graphs}
Let $\mmG=(\mmN, \mmE)$ be a directed graph (digraph) with node set $\mmN$ and edge set  $\mmE$. By abuse of  notation,  we write $e\in \mmG$ whenever $e\in \mmE$.
A \textbf{spanning tree} $\tau$ is a directed subgraph of $\mmG$ with node set $\mathcal{N}$ and such that the underlying undirected graph is connected and acyclic. A spanning tree is \emph{rooted} at the node $N$ if $N$ is the only node with no outgoing edges.

The graph $\mmG$ is \textbf{strongly connected} if there is a directed path from $N_1$ to $N_2$ for any pair of nodes $N_1$, $N_2$. Any directed path from $N_1$ to $N_2$ in a strongly connected graph can be extended to a spanning tree rooted at $N_2$. A \textbf{cycle} is a closed directed path $N_{i_1}\rightarrow N_{i_2}\rightarrow \cdots \rightarrow N_{i_n}\rightarrow N_{i_1}$ with no repeated nodes apart from the initial and terminal nodes. By definition all cycles are directed. 

If $\pi\colon\mmE \rightarrow R$ is a labeling of $\mmG$ with values in some ring $R$, then any subgraph $\mathcal{H}$ of $\mmG$ inherits a labeling from $\mmG$.  We extend the function $\pi$ to the set of subgraphs of $\mmG$ by defining
\[
\pi(\mathcal{H})=\prod_{e\in\mathcal{H}}\pi(e).
\]

A \textbf{multidigraph} $\mmG$ is a pair of finite sets $(\mmN, \mathcal{E})$ equipped  with two functions:
\begin{equation}\label{defmultigr}
s\colon\mathcal{E} \rightarrow \mmN\qquad t\colon\mathcal{E} \rightarrow \mmN.
\end{equation}
The elements of $\mmN$ are called nodes, the elements of $\mmE$ are called edges, and the functions 
$s,t$ are the source and target function, respectively.
The function $s$ assigns to each edge the source node of the edge and the function $t$ assigns to each edge the target node of the edge. In a multidigraph both self-edges (edges $e$ with $t(e)=s(e)$) and parallel edges (edges  $e_1,e_2$ with $t(e_1)=t(e_2)$ and $s(e_1)=s(e_2)$) are possible. 

Spanning trees, cycles and labels for a multidigraph are defined analogous to those of a digraph.  Note that the unique spanning tree of a multidigraph with one node and one self-edge is the node itself.

We  associate  a digraph $\widehat{\mmG}$ with a  multidigraph $\mmG$ by removing self-edges and collapsing 
parallel edges into one edge. 
That is, only one of the parallel edges between two nodes is kept. 
If  $\mmG$ is labeled, then so is $\widehat{\mmG}$ and the label of  an edge in $\widehat{\mmG}$ is the sum of the labels of the parallel edges  in $\mmG$ with the same source and target. If a  node in $\mmG$  is only connected to itself  then it is not included in  $\widehat{\mmG}$.
A formal definition of $\widehat{\mmG}$ is given in  Subsection \ref{sec:proofspreliminaries}.

\paragraph{\bf Reaction networks. }
A \textbf{reaction network} on a finite set $\mathcal{S}$ is a multidigraph $(\mmC, \mmR)$ where
\begin{enumerate}[(i)]
\item $\mathcal{S}=\{S_1,\dots,S_n\}$ is called the \textbf{species} set. It is equipped with an order that provides a canonical isomorphism $\R^{\mathcal{S}} \cong \R^n$.
\item $\mathcal{C}\subset \mathbb{R}_{\geq 0}^{n}$ is called the set of \textbf{complexes}.
\item $\mathcal{R}=\{r_1,\dots,r_{\ell}\}$ is called the set of \textbf{reactions}. 
\end{enumerate} 
The source function in \eqref{defmultigr} is denoted as $y\colon\mmR \rightarrow \mmC$ and assigns to each reaction $r_i$ its\textbf{ reactant} $y_{r_i}$. The target function in \eqref{defmultigr} is denoted as $y'\colon\mmR \rightarrow \mmC$ and assigns to each reaction $r_i$ its \textbf{product} $y'_{r_i}$.

Elements of $\R^{\mathcal{S}}$ are identified with linear combinations of species. Hence, under the isomorphism $\R^{\mathcal{S}} \cong \R^n$ in (i),  we write a complex as a linear combination of the species. For instance, the complex $(1,0,1)\in\mmC$, which corresponds to the element in $\R^{\mathcal{S}}$ assigning  $1$ to $S_1$ and $S_3$ and $0$ to $S_2$, is written as $S_1+S_3$.

It is assumed that $y_r\neq y'_r$ for all reactions $r\in\mmR$, and that every complex $\eta\in\mmC$ is  either the reactant or the product of some reaction $r\in\mmR$.  That is, the multidigraph  $(\mmC, \mmR)$ contains no self-edges and no isolated nodes.

Given a complex $\eta\in \mathcal{C}\subset \R^n_{\geq 0}$ we call  $\eta_i$  the \textbf{stoichiometric coefficient} of $S_i$ in $\eta$. 
We say that $\eta$ \textbf{involves} $S_i$ if $\eta_i\neq 0$,  and that a reaction $r$ \textbf{involves} $S_i$ if $S_i$ is involved in the reactant or product of $r$.
We say that a pair of species $S_i,S_j\in\mathcal{S}$, $i\neq j$, \textbf{interact} 
if they are both involved in the same reactant or product of a reaction. Equivalently, if $\eta_i,\eta_j\neq 0$, for some complex $\eta$.

It is not common to allow for multiple reactions between  the same reactant and product. It is however mathematically convenient for our purposes, as will be clear in Section \ref{reduction}. Note that, if $\mmG$ is a reaction network in our sense, then $\widehat{\mmG}$ is a reaction network in the standard sense.

 We often give a reaction network by listing its reactions. The set of complexes $\mmC$ and the set of species $\mmS$ is  easily found from the reactions.

\begin{example}\label{runex}
The following
\[
S_1+S_4 \ce{<=>} S_5 \ce{->} S_2+S_4 \ce{->} S_3+S_4
\]
is a reaction network on the set of species $\mmS=\{S_1,S_2,S_3,S_4,S_5\}$ with set of complexes
$\mmC=\{ S_1+S_4, S_5,S_2+S_4, S_3+S_4\}.$
\end{example}

\paragraph{\bf Dynamical systems. } 
A \textbf{kinetics} for a reaction network $(\mmC, \mmR)$ is a function
\[
\kappa\colon \Omega \rightarrow \R^{\ell}_{\geq 0}\qquad x\mapsto  (\kappa_{r_1}(x),\dots,\kappa_{r_{\ell}}(x)),
\]
where $\R^n_{>0}\subseteq \Omega\subseteq \R^n_{\geq 0}$, such that $\kappa(\R^n_{>0})\subseteq \R^{\ell}_{>0}$. 
The component $\kappa_{r_i}(x)$ of $\kappa(x)$ is called the \textbf{rate function} of the reaction $r_i\in \mmR$.

\textbf{Mass-action kinetics} is  defined by the following rate functions:
\[
\kappa_r(x)=k_rx^{y_r}=k_r\prod\limits_{i=1}^nx_i^{(y_r)_i},\qquad r\in\mmR,
\]
where $k_r>0$ is  called the reaction rate constant of the reaction $r$. By convention, $0^0=1$.

 The kinetics $\kappa$ provides a labeling of the reaction network.  
Since the set of reactions is ordered, we often denote the rate functions as $\kappa_i(x)$ instead of $\kappa_{r_i}(x)$.
In the examples, we label reactions with their rate functions.  In the particular case of mass-action kinetics, we simply use the reaction rate constants $k_i$ as labels.

We denote by $x_i$ the  concentration of species $S_i$ and let $x=(x_1,\dots ,x_n)$ be the vector of concentrations. In specific examples, the species are denoted by letters such as $E,P$ and  the concentrations are denoted by $x_E, x_P$, respectively. 

For a reaction network $(\mmC,\mmR)$ on $\mmS$ and a kinetics $\kappa(x)$,  we let 
\begin{equation}\label{eq:gi}
g_i(x)= \sum_{r\in \mathcal{R}}\kappa_r(x)(y'_r-y_r)_i,\qquad i=1,\dots,n, \quad x\in\Omega.
\end{equation}
The evolution of the species  concentration  in time is modelled by the following system of ODEs:
\begin{equation}\label{eq:ODE}
\dot{x}=g(x), \qquad x\in \Omega,
\end{equation}
where $\dot{x}=(\dot{x}_1,\dots,\dot{x}_n)$, $g(x)=(g_1(x),\dots,g_n(x))$, and where the derivative is with respect to time. Explicit reference to time is omitted.

The \textbf{steady states} of the system \eqref{eq:ODE} are the solutions to the system $$g(x)=0,\qquad x\in \Omega.$$ 

\begin{remark}\label{rk:hat}
Consider a reaction network $\mmG$ with a kinetics $\kappa$ and the associated reaction network $\widehat{\mmG}$ with the induced kinetics $\widehat{\kappa}$ (the induced labeling of $\widehat{\mmG}$).
Then the ODE system associated with $(\mmG,\kappa)$ agrees with the ODE system associated with $(\widehat{\mmG},\widehat{\kappa})$.
\end{remark}

\begin{remark}\label{rk:kineticszero}
In typical models of biochemical reaction systems, the rate function of a reaction  vanishes whenever the concentration of one of the reactant species is zero. In particular, this guarantees  invariance of  the non-negative orthant under \eqref{eq:ODE}. However, we do not need  this assumption for our results to hold.  In Section \ref{kinetics} we discuss some results that follow from making this assumption.
\end{remark}

\paragraph{\bf Conservation laws. }
The \textbf{stoichiometric subspace} of a network $(\mmC,\mmR)$ is the  vector subspace of $\mathbb{R}^n$  given by
\[
S=\langle y'_r-y_r\st r\in \mathcal{R}\rangle\subset \R^n.
\]
If $\omega=(\omega_1,\dots,\omega_n)\in S^{\bot}$, then it follows from  \eqref{eq:gi}  and \eqref{eq:ODE} that $\omega\cdot \dot{x}=0$. Thus, for any  trajectory there is a constant $T\in\R$ such that 
$$T=\omega\cdot x= \sum\limits_{i=1}^n\omega_ix_i.$$ 
This equation is called the \textbf{conservation law} with total amount $T\in \R$, corresponding to $\omega\in S^{\bot}$. We also say that 
$\omega\cdot x$ is \textbf{conserved}. 
A set of conservation laws  
\[
\{ T_1=\omega^1\cdot x, \ \dots \ ,\ T_l=\omega^l\cdot x \}
\]
is \textbf{minimal} if  $\omega^1,\dots,\omega^l$ form a basis of $S^{\bot}$. Then, the trajectory with initial concentration $x_0$  is confined to the linear space with equations
\[
T_i=\omega^i\cdot x, \quad \text{with }\quad T_i=\omega^i\cdot x_0, \quad \text{for }i=1,\dots,l.
\]

\begin{example}\label{runex:2}
Consider Example \ref{runex} and let $\kappa$ be a kinetics:
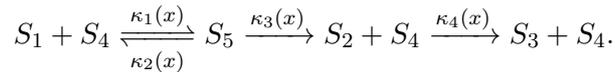
\[
S_1+S_4 \ce{<=>[\kappa_1(x)][\kappa_2(x)]} S_5 \ce{->[\kappa_3(x)]} S_2+S_4 \ce{->[\kappa_4(x)]} S_3+S_4.
\]
The corresponding ODE system  is:
\begin{align*}
\dot{x}_1=&-\kappa_{{1}}(x)+\kappa_{{2}}(x), & \dot{x}_2=&\kappa_{{3}}(x)-\kappa_{{4}}(x), &   \dot{x}_3=&\kappa_{{4}}(x),\\
\dot{x}_4=&-\kappa_{{1}}(x)+\kappa_{{2}}(x)+\kappa_{{3}}(x),  &
\dot{x}_5=&\kappa_{{1}}(x)-\kappa_{{2}}(x)-\kappa_{{3}}(x),
\end{align*}
and a minimal set of conservation laws consists of
\begin{equation}\label{CLrunex}
x_1+x_2+x_3+x_5=T_1,\hspace{15pt} x_4+x_5=T_2.
\end{equation}
Since concentrations are nonnegative, $T_1,T_2\geq 0$.
\end{example}

\section{Elimination of variables}\label{elimination}

In this section we introduce some  results from \cite{Fel_elim} about linear elimination of variables.
The goal in \cite{Fel_elim} is to express the concentrations of some  of the species at 
steady state in terms  of the concentrations of the other species.  This is done for sets of species, say $\mmU$, that form a noninteracting set (Definition \ref{def:nonint}), assuming  the kinetics   is linear with respect to the  concentrations of the species in $\mmU$ (Definition \ref{defUlinear}). 
The elimination makes use of a special multidigraph, called  $\mmG_{\mmU}$, which is defined in Definition \ref{defgraph} (see also Theorem \ref{elim}). This multidigraph is also  used to define the reduced reaction network in Section \ref{reduction}.

Let a  reaction network $(\mmC,\mmR)$ on a set $\mmS$ be given. 

\begin{definition}\label{def:nonint}
A subset $\mmU\subset \mmS$ is \textbf{noninteracting} if it contains no pair of interacting species, and the stoichiometric coefficients of the species in $\mmU$ in all complexes are either $0$ or $1$.
\end{definition}

Let $\mmU\subseteq \mmS$ be a noninteracting subset of species. For simplicity we let $\mmU=\{U_1,\dots ,U_m\}$ and  $\mmU^c=\mmS\setminus \mmU=\{S_1,\dots,S_p\}$  (with $p=n-m$), such that $\mmS=\mmU^c\cup \mmU$. We order   $\mmS$ as $\mmS=\{S_1,\dots,S_p, U_1,\dots,U_m\}$. From now on, we let $x_i$ be the  concentration of $S_i\in\mmU^c$ and $x=(x_1,\dots,x_p)$.  
Similarly, we let $u_i$ be the  concentration of $U_i\in\mmU$ and $u=(u_1,\dots, u_m)$. Hence, the rate functions are functions of $(x,u)$: $\kappa_r(x,u)$.

We let 
\begin{equation}
\label{projection}
\zeta\colon \R^n\rightarrow \R^p,\qquad \rho\colon\R^n\rightarrow \R^m
\end{equation}
be the projections onto the first $p$ coordinates and  last $m$ coordinates of $\R^n$, respectively.
Further, we  let
$\mmRu$ be the set of reactions that involve species in $\mmU$  in the reactant and/or in the product:
\[
\mmRu=\{r\in \mmR \st   \rho(y_r)\neq 0 \text{ or }  \rho(y'_r)\neq 0\}.
\]
Any reaction in $\mmR_\mmU$ involves at most one species in $\mmU$ in the reactant and at most one in the product. Hence,  the vectors $\rho(y_r)$ and $\rho(y'_r)$ have at most one nonzero component (equal to one) for any  $r\in\mmRu$.

We next impose some regularity conditions on the rate functions, namely,  that the rate functions  are linear in the concentrations of the species in $\mmU$ in a specific way. We assume $\Omega$ takes the form $\Omega\times\R_{\ge 0}^m$, where $\R^p_{>0}\subseteq \Omega\subseteq \R^p_{\geq 0}$ (making an abuse of notation).

\begin{definition}\label{defUlinear}
Let $\kappa\colon \Omega\times \R^m_{\geq 0}\rightarrow \R^{\ell}_{\geq 0}$ be a kinetics with $\R^p_{>0}\subseteq \Omega\subseteq \R^p_{\geq 0}$. The kinetics $\kappa$ is \textbf{$\mmU$-linear} if, for each $r\in \mmRu$, there exists a function $v_r\colon\Omega\rightarrow \R_{\geq 0}$ such that $v_r(x)> 0$ for all $x\in\Omega$ and
\[
\kappa_r(x,u)=\left\{\begin{array}{ll} 
u_iv_r(x) & \text{ if }r\in\mmRu\text{ and }\rho(y_r)_{i}=1\\
v_r(x)& \text{ if }r\in\mmRu \text{ and }\rho(y_r)= 0.
\end{array}\right.
\]
\end{definition}

A $\mmU$-linear kinetics can be interpreted as being mass-action with respect to the species in $\mmU$.

\begin{definition}\label{defgraph} 
Let $\mmU\subseteq \mmS$ be a set of noninteracting species and $\kappa$ a $\mmU$-linear kinetics.
We define the labeled \textbf{multidigraph} $\mmG_{\mmU}=(\mmN_{\mmU}, \mmE_{\mmU})$ by 
\[
\mmN_{\mmU}=\begin{cases}
\mmU& \text{if } \rho(y_r)\neq 0\text{ and }  \rho(y'_r)\neq 0 \text{ for all }r\in \mmRu\\
\mmU\cup \{*\} & \text{otherwise}
\end{cases} 
\]
and
\begin{align*}
\mathcal{E}_\mmU= & \ \{U_i\ce{->[v_r(x)]}U_j\st r\in \mmRu\text{ with }\rho(y_r)_{i}\neq 0\text{ and }  \rho(y'_r)_{j}\neq 0\}\ \cup \\
	& \ \{U_i\ce{->[v_r(x)]}*\st r\in \mmRu\text{ with }\rho(y_r)_{i}\neq 0\text{ and }  \rho(y'_r)= 0\}\ \cup \\
	& \ \{*\ce{->[v_r(x)]}U_i\st r\in \mmRu\text{ with }\rho(y'_r)_{i}\neq 0\text{ and }  \rho(y_r)= 0\},
\end{align*}
 where the last two subsets in the definition of $\mathcal{E}_\mmU$ are empty when $\mmN_{\mmU}=\mmU$.
\end{definition}

Observe that $\mmG_{\mmU}$ might contain parallel edges between any pair of nodes and self-edges for nodes other than $*$. Each edge $e$ in the multidigraph $\mmG_{\mmU}$ corresponds to a reaction in $\mmR_\mmU$, $r(e)$. Moreover, since $\mmU$ is  noninteracting, the multidigraph $\mmG_{\mmU}$ has exactly one edge $e(r)$ for each reaction $r$ in $\mmRu$.
This gives rise to two bijective functions 
\begin{equation}\label{defre}
\xymatrix{
\mmE_\mmU \ar@<.5ex>[r]^-{r} & \mmRu \ar@<.5ex>[l]^-{e},
}
\end{equation}
such that $r\circ e=\text{id}_{\mmRu}$ and $e\circ r=\text{id}_{\mmE_\mmU}$.
In the examples, the functions $r,e$ are implicitly given by the 
subindices of the  functions $v_r(x)$: the edge label $v_i(x)$ indicates that the edge corresponds to the reaction $r_i$.

The graph $\widehat{G}_\mmU$ (not taking into account the labelling) agrees with the embedded network of $(\mmC,\mmR)$ given by $\mmU$, as defined in \cite{joshi-shiu-II} considering $*$ as the zero complex.

\begin{example}\label{noncutexample}
Consider the network 
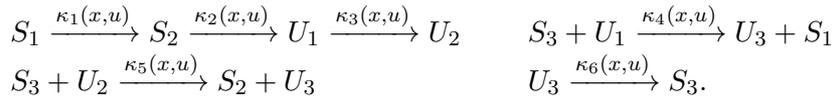
\[
\begin{array}{ll}
S_1\ce{->[\kappa_1(x,u)]}S_2\ce{->[\kappa_2(x,u)]}U_1\ce{->[\kappa_3(x,u)]} U_2 &\hspace{15pt}
S_3+U_1\ce{->[\kappa_4(x,u)]} U_3+ S_1 \\
 S_3+U_2\ce{->[\kappa_5(x,u)]}S_2+U_3 &\hspace{15pt} U_3 \ce{->[\kappa_6(x,u)]}S_3.
\end{array}
\]
The set  $\mmU=\{U_1,U_2,U_3\}$  is  noninteracting.
A  $\mmU$-linear kinetics  $\kappa(x,u)$  takes the form
\begin{align*}
\kappa(x,u) &= \big( \kappa_1(x,u), v_2(x),  u_1v_3(x), u_1v_4(x),  u_2v_5(x), u_3v_6(x) \big)
\end{align*}
and the corresponding multidigraph $\mmG_{\mmU}$ is 
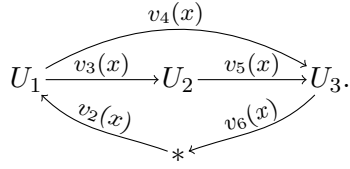
\begin{center}
\begin{tikzpicture}[inner sep=1pt]
\node (X3) at (0,-2) {$U_1$};
\node (X4) at (2,-2) {$U_2$};
\node (X6) at (4,-2) {$U_3$.};
\node (*) at (2,-3) {$*$};
\draw[->] (X3) to[out=0,in=180] node[above,sloped] {\footnotesize $v_3(x)$} (X4);
\draw[->] (X4) to[out=0,in=180] node[above,sloped] {\footnotesize $v_5(x)$} (X6);
\draw[->] (X3) to[out=30,in=150] node[above,sloped] {\footnotesize $v_4(x)$} (X6);
\draw[->] (*) to[out=160,in=-45] node[above,sloped] {\footnotesize $v_2(x)$} (X3);
\draw[->] (X6) to[out=225,in=20] node[above,sloped] {\footnotesize $v_6(x)$} (*);
\end{tikzpicture}
\end{center}
\end{example}

\begin{example} \label{2c1r}
Consider a phosphorylation mechanism for a substrate $S$ such that phosphorylation may be catalyzed by two different enzymes, $E_1$ and $E_2$, with mass-action kinetics:
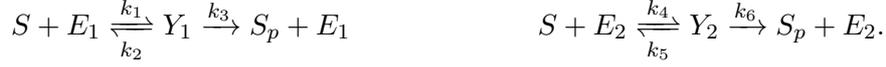
\begin{align*}
S+E_1 &  \ce{<=>[k_1][k_2]}  Y_1  \ce{->[k_3]}  S_p+E_1& \quad S+E_2  &\ce{<=>[k_4][k_5]}  Y_2  \ce{->[k_6]}  S_p+E_2.
\end{align*}
We consider the noninteracting set $\mathcal{U}=\{E_1,E_2,Y_1,Y_2\}$. Since we assumed mass-action kinetics, the kinetics is $\mmU$-linear. The  multidigraph $\mmG_{\mathcal{U}}$ is 

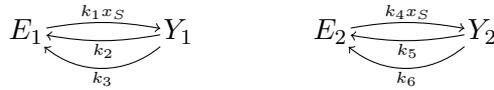
\begin{center}
\begin{tikzpicture}[inner sep=1pt]
\node (E1) at (-2,0) {$E_1$};
\node (E2) at (2,0) {$E_2$};
\node (Y1) at (0,0) {$Y_1$};
\node (Y2) at (4,0) {$Y_2$};
\draw[->] (E1) to[out=10,in=170] node[above,sloped] {\footnotesize $\scriptstyle k_1x_S$} (Y1);
\draw[->] (Y1) to[out=190,in=-10] node[below,sloped] {\footnotesize $\scriptstyle k_2$} (E1);
\draw[->] (Y1) to[out=220,in=-40] node[below,sloped] {\footnotesize $\scriptstyle k_3$} (E1);
\draw[->] (E2) to[out=10,in=170] node[above,sloped] {\footnotesize $\scriptstyle k_4x_S$} (Y2);
\draw[->] (Y2) to[out=190,in=-10] node[below,sloped] {\footnotesize $\scriptstyle k_5$} (E2);
\draw[->] (Y2) to[out=220,in=-40] node[below,sloped] {\footnotesize $\scriptstyle k_6$} (E2);
\end{tikzpicture}
\end{center}
\end{example}

Define the \textbf{support} of  a vector $\omega=(\omega_1,\dots,\omega_n)\in \R^n$ as
 $$\supp(\omega)=\{S_i\st \omega_i\neq 0\}\subseteq \mmS.$$

 A subset $\mmH\subseteq \mmU$ of a noninteracting set is itself a noninteracting set and a $\mmU$-linear kinetics is $\mmH$-linear. Hence the multidigraph $\mmG_\mmH$ is defined.
 
Any connected component of  $\mmG_\mmU$ agrees with $\mmG_\mmH$ for some $\mmH\subseteq \mmU$.
If $*\not\in\mmN_\mmH$,
then the reactant and the product of any reaction in $\mmR_\mmH$    involve each exactly one species  in $\mmH$. 
In this case there is a vector   
$\omega^\mmH \in S^\perp$ such that  
\begin{equation}\label{eq:conscut} 
\zeta(\omega^\mmH)=0 \qquad \text{and}\qquad\rho(\omega^\mmH)_i = \begin{cases}
0 & \text{if } U_{i} \notin \mmH \\ 1 & \text{if } U_{i} \in \mmH
\end{cases}\quad i=1,\dots,m,
\end{equation}
that is, the coordinates of $\omega^\mmH$ corresponding to the species in $\mmH$ are all one and only these are non-zero. 

According to \cite{Fel_elim},  any vector in $S^\bot$ with support in $\mmU$ is a linear combination of 
 the vectors $\omega^\mmH$ for all connected components $\mmG_\mmH$  such that $*\not\in\mmN_\mmH$.
Since these connected components 
are disjoint, the vectors $\omega^\mmH$ are linearly independent.
These vectors give rise to   \textbf{conservation laws}:
\begin{equation}\label{eq:conslawU}
\sum_{i\st U_i\in \mmH } u_i = T_\mmH.
\end{equation}

In Example~\ref{2c1r}, there are two connected components without the node $*$. They give rise to two conservation laws, namely,  $x_{E_1}+ x_{Y_1} =T_1$ and $x_{E_2}+ x_{Y_2} =T_2$. Any other conservation law involving only the concentrations of $E_1,E_2,Y_1,Y_2$ is a linear combination of these two conservation laws. The only connected component of $\mmG_\mmU$ in Example~\ref{noncutexample} contains $*$. Hence there is no conservation law involving only concentrations of species in $\mmU$.

Let $N$ be a node of  $\mmG_\mmU$ and $\mmG_\mmH$ be the connected component that contains it, that is $N\in \mmN_\mmH$. Let $\Theta(N)$ be the set of spanning trees of  $\mmG_{\mmH}$  rooted at $N$.
\begin{definition}\label{def:linelim}
Let $(\mmC,\mmR)$ be a reaction network on a set   $\mmS$  and $\kappa$ a kinetics. A set  $\mmU\subseteq\mmS$  is said to be \textbf{linearly eliminable} if
\begin{enumerate}[(i)]
\item $\mmU$ is noninteracting.
\item For every connected component $\mmG_\mmH$ of $\mmG_\mmU$ it holds:  $\Theta(*)\neq \emptyset$ if $*\in \mmN_\mmH$ and   $\cup_{N\in\mmH} \Theta(N)\not=\emptyset$ if $*\not\in \mmN_\mmH$. 
\item $\kappa$ is $\mmU$-linear.
\end{enumerate}
\end{definition}
Condition (ii) states that $\mmG_\mmH$ admits a spanning tree   rooted at $*$, if $*\in \mmN_\mmH$, and  rooted at an arbitrary  node, if $*\not\in \mmN_\mmH$. Any strongly connected component of $\mmG_\mmU$  fulfills this condition. 

 Let $\mmU\subseteq\mmS$ be a linearly eliminable set. 
For each connected component $\mmG_\mmH$ of $\mmG_{\mmU}$, let  $q_\mmH(x)$ be the function
\begin{align}
&q_\mmH(x)=\frac{T_\mmH}{\sum\limits_{N \in \mmH,\tau \in \Theta(N)}\pi(\tau) }, && \text{if }*\not\in\mmN_\mmH, 
\text{ where }T_\mmH\in \R_{\geq 0},\label{qcut}\\
&q_\mmH(x)=\frac{1}{\sum\limits_{\tau \in \Theta(*)}\pi(\tau)}, &&\text{if }*\in\mmN_\mmH. 
\label{qnoncut}
\end{align}
The function $q_\mmH(x)$ is a function of $x$   since $\pi(\tau)$  depends on $x$. We omit the explicit reference to $T_\mmH$ in $q_\mmH(x)$ and to $x$ in $\pi(\tau)$ for convenience. 
When the multidigraph $\mmG_{\mmU}$ is connected, we omit the subscript $\mmH=\mmU$.

Definition~\ref{def:linelim}(ii) guarantees that the denominators are never empty sums. By positivity of $v_r(x)$, the function $q_\mmH(x)$ is positive for all $x\in \R^p_{>0}$. However, 
the denominator of $q_\mmH(x)$ 
might vanish at  points at the boundary of $\Omega$. Henceforth, we let   $\widetilde{\Omega}$  be the domain where $q_\mmH(x)$ is well defined for all connected components $\mmG_\mmH$.  We have $\R^p_{>0}\subseteq\widetilde{\Omega}\subseteq \Omega$.

We will make use of \cite[Propositions 8.4, 8.6]{Fel_elim}, but reformulate them in the terminology of this paper
(see also Section \ref{sec:proofselimination}).

\begin{theorem}[Elimination \cite{Fel_elim}]
\label{elim}
Let $(\mmC,\mmR)$ be a  reaction network on a set $\mmS$, $\kappa$  a kinetics and assume  $\mmU\subseteq \mmS$ is  a linearly eliminable set. 
Consider the system of equations
\begin{equation}\label{eq:uss}
\dot{u}  =0, \quad \text{and}\quad  T_\mmH  = \sum_{i\st U_i\in \mmH } u_i, \quad T_\mmH \ge 0,
\end{equation}
for all $\mmH\subseteq \mmU$ such that $\mmG_\mmH$ is a connected component of $\mmG_\mmU$ with $*\notin \mmN_\mmH$.
Equation \eqref{eq:uss}  holds for  $(x,u)\in\widetilde{\Omega}\times \R^m_{\geq 0}$
 if and only if 
\[
u_i=q_{\mmH(i)}(x)\sum\limits_{\tau \in \Theta(U_i)}\pi(\tau), \hspace{20pt}\text{for all } i=1\dots m,
\]
where $q_{\mmH(i)}(x)$ is given by \eqref{qcut} or \eqref{qnoncut}, and $ \mmH(i)\subseteq \mmU$ is such that $U_i\in\mmH(i)$ for $\mmG_{\mmH(i)}$ a connected component of $\mmG_\mmU$.
\end{theorem}

The theorem states that if the species in $\mmU$ are at steady state with respect to the species in $\mmU^c$, then  the concentrations $u_i$ can be given as rational functions in the labels of $\mmG_{\mmU}$, for fixed total amounts $T_\mmH$. 
The labels are independent of the  concentrations  of the species in $\mmU$. In virtue of the non-negativity of the labels of $\mmG_{\mmU}$, the  concentrations of the species in $\mmU$ are also non-negative for $x\in\widetilde{\Omega}$. 

For future reference, we define the function $\varphi\colon  \widetilde{\Omega}\to \R_{\ge 0}^m$:
\begin{equation}\label{eq:varphi}
 \varphi(x)=(\varphi_1(x),\dots,\varphi_m(x)),\qquad \varphi_i(x)= q_{\mmH(i)}(x)\hspace{-10pt}\sum\limits_{\tau \in \Theta(U_i)}\pi(\tau). 
 \end{equation}

If a connected component  $\mmG_{\mmH}$ of $\mmG_{\mmU}$ is strongly connected, then there is at least one spanning tree rooted at each node, and hence $\varphi(x)$ is not identically zero in any coordinate corresponding to a species in $\mmH$.
If  $\mmG_{\mmH}$ is not strongly connected, then at least one  node is not the root of a spanning tree, and hence at least one  component of $\varphi(x)$ is identically zero. In the latter case, all steady states lay at the boundary of the positive orthant.

\begin{example}\label{runexelim}
Consider the network in Example \ref{runex} and the noninteracting set $\mmU=\{S_4,S_5\}$. 
To adapt the notation to the present setting, we let $U_1=S_4$ and $U_2=S_5$. 
We assume the network is endowed with a $\mmU$-linear kinetics  
\begin{align*}
\kappa(x,u) & = \big( u_1 v_1(x),  u_2v_2(x), u_2v_3(x), u_1v_4(x) \big).
\end{align*}
The multidigraph $\mmG_{\mmU}$ is
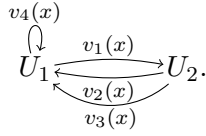
\begin{center}
\begin{tikzpicture}[inner sep=1pt]
\node (X2) at (0,-2) {$U_1$};
\node (X3) at (2,-2) {$U_2$.};
\draw[->] (X2) to[out=10,in=170] node[above,sloped] {\scriptsize $v_1(x)$} (X3);
\draw[->] (X3) to[out=190,in=-10] node[below,sloped] {\scriptsize $v_2(x)$} (X2);
\draw[->] (X3) to[out=220,in=-40] node[below,sloped] {\scriptsize $v_3(x)$} (X2);
\draw[->] (X2) edge [loop above] node[above,sloped] {\scriptsize $v_4(x)$} (X2);
\end{tikzpicture}
\end{center}
The multidigraph $\mmG_\mmU$  is strongly connected and thus the set $\mmU$ is linearly eliminable. We apply Theorem \ref{elim} to conclude that $\dot{u}_1=0$, $\dot{u}_2=0$ and $u_1+u_2=T_2$ hold if and only if:
\[
u_{1}=q(x) \big( v_{3}(x)+v_{2}(x) \big), \quad 
u_{2}=q(x) v_{1}(x), \quad q(x)= \dfrac {T_2 }{v_{{1}}(x)+v_{{2}}(x)+v_{{3}}(x)}.
\]
\end{example}

\begin{example}
The set $\mmU=\{U_1,U_2,U_3\}$ in Example \ref{noncutexample} endowed with $\mmU$-linear kinetics is linearly eliminable since 
 $\mmG_\mmU$  is strongly connected. By Theorem \ref{elim}, at steady state it holds that:
\[
\arraycolsep=10pt
\begin{array}{ll}
u_1=q(x) v_2(x)v_5(x)v_6(x), & u_2=q(x) v_2(x)v_3(x)v_6(x),  \\ [5pt]
u_3=q(x) v_2(x)(v_3(x)+v_4(x))v_5(x),&q(x)=\dfrac{1}{(v_3(x)+v_4(x))v_5(x)v_6(x)}.
\end{array}
\]
\end{example}

\section{Reduced reaction network}\label{reduction}

Assume that a reaction network $(\mmC,\mmR)$ on a set $\mmS$, a kinetics $\kappa$ and a linearly eliminable set  $\mmU\subseteq \mmS$  are given. 
If $u=\varphi(x)$, as given in  \eqref{eq:varphi}, is substituted into the ODE system $(\dot{x}, \dot{u})=g(x,u)$, given in    \eqref{eq:ODE},  then
$\dot{u}=0$ by construction and  an ODE system in the variables $x_1,\dots,x_p$ is obtained:
\begin{equation}\label{eq:newODE}
\dot{x}=\tilde{g}(x),\qquad x\in \widetilde{\Omega}\subseteq \R^p_{\geq 0},
\end{equation}
where
\[
\tilde{g}(x) =\zeta (g(x,\varphi(x))).
\]
That is, for $i=1,\dots,p$,
\[
\tilde{g}_i(x) =g_i\left(x_1,\dots, x_p,\ q_{\mmH(1)}(x)\hspace{-6pt}\sum\limits_{\tau \in \Theta(U_1)}\!\!\!\pi(\tau), \dots, q_{\mmH(m)}(x)\hspace{-6pt}\sum\limits_{\tau \in \Theta(U_m)}\!\!\!\pi(\tau)\right).
\]

In this section we prove that the   ODE system \eqref{eq:newODE}  is the ODE system associated with a reaction network on the species set  $\{S_1,\dots,S_p\}$. The reactions   and the kinetics of this  reaction network are graphically obtained from the  multidigraph $\mmG_{\mmU}$, following the  procedure given below Theorem \ref{rednet}.

The intuition is as explained in the introduction. Cycles in the multidigraph correspond to sequences of reactions such that the net production of the species in $\mmU$ is zero. However, not all cycles give rise to a reduced reaction. 
In order to give the precise formulation of the reduced reaction network (Definition \ref{defrednet}, Theorem \ref{rednet}), we will need the following. 
For a cycle $\sigma$  in $\mmG_\mmU$, let $\mmG_{\mmH(\sigma)}$ be the connected component that contains the cycle. In particular, the nodes of  $\sigma$ form a subset of $\mmH(\sigma)$. Let $\Gamma(\sigma)$ be the set of subgraphs $G$ of $\mmG_{\mmU}$ such that $\sigma\subseteq G$   and such that removing any edge $e\in\sigma$ from $G$ creates a spanning tree of $\mmG_{\mmH(\sigma)}$, rooted at the source $s(e)$ of $e$.
 Specifically, given an edge $e\in \sigma$,  define
\begin{equation}\label{eq:gammasigma} 
\Gamma(\sigma) = \{ \tau \cup e \st \sigma\setminus e \subset \tau, \ \tau\in \Theta(s(e)) \}.
\end{equation}
The set $\Gamma(\sigma)$ does not depend on the chosen edge $e$ (Proposition \ref{constGamma} in Section \ref{sec:proofsrednet}).
We further define the function
\begin{equation}\label{defPi}
\Pi(\sigma)= \sum\limits_{\gamma\in \Gamma(\sigma)} \pi(\gamma) = \pi(\sigma)\sum\limits_{\gamma\in \Gamma(\sigma)} \pi(\gamma\setminus\sigma),
\end{equation}
which is the sum of the labels of the graphs in $\Gamma(\sigma)$.

\begin{remark}\label{rmk:compGammasigma}
To find $\Gamma(\sigma)$, it is sufficient to consider any edge $e\in\sigma$ and find all spanning trees rooted at $s(e)$, containing $\sigma\setminus e$. Therefore, $\Gamma(\sigma)\neq \emptyset$ if and only if the path $\sigma\setminus e$ can be extended to a spanning tree of $\mmG_{\mmH(\sigma)}$ rooted at $s(e)$. In particular, if $\mmG_{\mmH(\sigma)}$ is strongly connected, then $\Gamma(\sigma)\neq \emptyset$.

If $\sigma$ is a cycle that contains all nodes of $\mmG_{\mmH(\sigma)}$, then $\Gamma(\sigma)=\{\sigma\}$ and thus $\Pi(\sigma)=\pi(\sigma)$. Indeed, in this case  $\sigma\setminus e$ is the unique spanning tree of $\mmG_{\mmH(\sigma)}$ containing $\sigma\setminus e$.
\end{remark}

Let $\Delta$ be the  set of cycles $\sigma$ in $\mmG_\mmU$ such that 
 $$\sum\limits_{e\in\sigma}\zeta(y'_{r(e)}-y_{r(e)})\neq 0\quad\textrm{and}\quad\Gamma(\sigma)\neq \emptyset.$$ 
The first condition states that the net-production of some species in $\mmU^c$ is non-zero over the reactions associated with the edges of the cycle.

\begin{definition}\label{defrednet}
 Let $(\mmC,\mmR)$ be a reaction network on a set  $\mmS$, 
 $\kappa$ a kinetics and $\mmU\subseteq \mmS$ a linearly eliminable set.
The \textbf{reduced reaction network obtained by elimination of $\mmU$} is the reaction network $(\widetilde{\mmC}, \widetilde{\mmR})$ on the  species set $\mmU^c$ with kinetics $\widetilde{\kappa}$ defined on $\widetilde{\Omega}$, such that $\widetilde{\mmR}  = \wR_1\cup \wR_2$, where
\begin{align*}
\wR_1 &= \left\{\widetilde{r}\colon\zeta(y_r)\rightarrow \zeta(y'_r)\st r\in \mmR\setminus \mmRu \right\}, &\\
\wR_2 & =\left\{\widetilde{r}_{\sigma}\colon \sum\limits_{e\in\sigma}\zeta(y_{r(e)})\rightarrow \sum\limits_{e\in\sigma}\zeta(y'_{r(e)})\st \sigma\in \Delta \right\},
\end{align*}
$\widetilde{\mmC}$ is the set of source and target nodes of $\widetilde{\mmR}$, 
\[
\widetilde{\kappa}_{\widetilde{r}}(x):= \begin{cases}
\kappa_r(x,\varphi(x)) & \textrm{if } \widetilde{r}=\zeta(y_r)\rightarrow \zeta(y'_r)\in \wR_1,\ \textrm{for } r\in \mmR\setminus \mmRu
, \\
q_{\mmH(\sigma)}(x)\Pi(\sigma) &  \textrm{if } \widetilde{r} = \widetilde{r}_{\sigma} \in \wR_2.
 \end{cases}
\]
and $\widetilde{\Omega}\subseteq \Omega$ is the maximal set for which $q_\mmH(x)$ is well defined for all connected components $\mmG_\mmH$ of $\mmG_\mmU$. 
\end{definition}

In short, $\widetilde{\mmR}_1$ consists of the reactions that do not involve species in $\mmU$ and $\widetilde{\mmR}_2$ consists of new reactions defined by the cycles in $\Delta$. The procedure outlined below Theorem~\ref{rednet} clarifies how the reactions in $\widetilde{\mmR}_2$ are obtained.

We are now ready to present our main result:
the interpretation of the ODE system \eqref{eq:newODE} as the ODE system of the reduced reaction network obtained by elimination of $\mmU$ defined above. The proof is given in Subsection~\ref{sec:proofsrednet}.

\begin{theorem}[{\bf Reduced reaction network}]\label{rednet}
Let $(\mmC,\mmR)$ be a reaction network on a set $\mmS$, 
 $\kappa$ a kinetics and $\mmU\subseteq \mmS$ a linearly eliminable set.
Then the  ODE system in \eqref{eq:newODE}, 
$$\dot{x}_i=g_i(x,\varphi(x)),\qquad x\in \widetilde{\Omega}\subseteq \R^p_{\geq 0},\quad i=1,\dots,p,
$$
is the ODE system associated with the reduced reaction network obtained by elimination of $\mmU$.
\end{theorem}

The fact that we can find a reaction network for the reduced ODE system implies that the system does not have negative cross effects in the sense of \cite{Szili}. 

\medskip
\noindent
{\bf Procedure. }  
We provide the following procedure to determine the reduced reaction network obtained by elimination of a set of species  $\mmU$, in virtue of Definition~\ref{defrednet} and Theorem~\ref{rednet}:
\begin{enumerate}[(1)]
\item Check whether  $\mmU$ is linearly eliminable: 
\begin{enumerate}[(i)] 
\item Check that the set $\mmU$ is noninteracting and that $\kappa$ is $\mmU$-linear.
\item Compute the multidigraph $\mmG_{\mmU}$ (as in Definition \ref{defgraph}) and check that each connected component $\mmG_\mmH$  of $\mmG_\mmU$ admits a spanning tree  rooted at an arbitrary node, 
if $*$ is not a node of $\mmG_\mmH$ and rooted at $*$ otherwise.
\end{enumerate}
If either (i) or (ii) fails, stop.
\item *Compute $q_\mmH(x)$ for each connected component $\mmG_\mmH$  of $\mmG_\mmU$  and $\varphi(x)$, using \eqref{qcut}, \eqref{qnoncut} and \eqref{eq:varphi}.
\item Compute the reactions in $\wR_1$: for each reaction that is not in $\mmR_{\mmU}$ do
\begin{enumerate}[(i)] 
\item Add the projection of the reaction by $\zeta$ to $\wR_1$.
\item *Compute its rate function by replacing $u$ by $\varphi(x)$ in the original rate function.
\end{enumerate}
\item Compute the reactions in $\wR_2$: for each cycle $\sigma$ in $\mmG_{\mmU}$ do
\begin{enumerate}[(i)]
\item \label{cancelincycle} List the reactions corresponding to the edges in the cycle and add the species in $\mmU^c$ in  the reactants and products  to obtain the reactant and product of the new candidate reaction $\widetilde{r}_\sigma$. 
\item \label{reactprod}  If $y_{\widetilde{r}_\sigma} \neq y'_{\widetilde{r}_\sigma}$ proceed. Otherwise go to the next cycle.
\item \label{gammasigma} Compute $\Gamma(\sigma)$. If $\Gamma(\sigma)\neq \emptyset$, add $\widetilde{r}_\sigma$ to $\wR_2$ and proceed. Otherwise go to the next cycle.
\item *The rate function of the new reaction is  $q_{\mmH(\sigma)}(x)\Pi(\sigma)$, cf. \eqref{defPi}.
\end{enumerate}
\end{enumerate}

If we are only interested  in  the reactions of the reduced reaction network and not  the kinetics, then the steps marked with  * can be ignored. Further, in this case, if $\mmG_{\mmH(\sigma)}$ is strongly connected, then $\Gamma(\sigma)\neq \emptyset$ and we do not need to find $\Gamma(\sigma)$ explicitly.

\begin{remark}\label{trivialreact}
Let $r_{i_1},\dots, r_{i_l}\in\mmR$ be reactions  in a reaction network $(\mmC, \mmR)$ that form a cycle in the network. Then, 
\[
\sum\limits_{j=1}^l\big(y'_{r_{i_j}}-y_{r_{i_j}}\big)=0.
\]
Hence, step \eqref{reactprod} is not fulfilled and 
the cycle does not belong to $\Delta$. 
This applies in particular to any cycle arising from a reversible reaction $\eta_1\ce{<=>}\eta_2$.
\end{remark}

\begin{example}\label{ex:runningcycles}
For illustration,  we apply the method to the reaction network in Example \ref{runex}, with $\mmU=\{S_4,S_5\}=\{U_1,U_2\}$ and a $\mmU$-linear kinetics.
\begin{enumerate}[(1)]
\item $\mmU$ is  linearly eliminable: see Example~\ref{runexelim}.
\item  $q(x)$ and $\varphi(x)$ are given in Example~\ref{runexelim}.
\item $\mmR_\mmU=\mmR$. Therefore $\wR_1=\emptyset$.
\item We compute the reactions in $\wR_2$. $\mmG_\mmU$ has three cycles:

\medskip
 
\begin{minipage}[h]{0.22\linewidth}
\begin{tikzpicture}[inner sep=1pt]
\node (X2) at (0,0) {$U_1$};
\node (X3) at (2,0) {$U_2$};
\node (s) at (0,0.5) {$\sigma_1$:};
\draw[->] (X2) to[out=10,in=170] node[above,sloped] {\scriptsize $v_1(x)$} (X3);
\draw[->] (X3) to[out=220,in=-40] node[below,sloped] {\scriptsize $v_3(x)$} (X2);
\end{tikzpicture}
\end{minipage} \qquad
\begin{minipage}[h]{0.69\linewidth}
\begin{enumerate}[(i)]
\item 
  \qquad $ 
\begin{array}[t]{cccc}
r_1\colon& S_1+\cancel{U_1} &\ce{->} & \cancel{U_2}\\
r_3\colon& \cancel{U_2} &\ce{->} &  S_2+\cancel{U_1} \\
\hline
\widetilde{r}_\sigma\colon & S_1 &\ce{->} &S_2
\end{array}
$
\item $y_{\widetilde{r}_\sigma} = S_1 \neq S_2 = y'_{\widetilde{r}_\sigma}$.
\item $\Gamma(\sigma_1) = \{ \sigma_1\} \neq \emptyset$ (see also Remark~\ref{rmk:compGammasigma}).
The reaction $S_1\ce{->} S_2$ belongs to $\wR_2$.
\item Rate function: $q(x)v_1(x)v_3(x)$
\end{enumerate}

\end{minipage} 

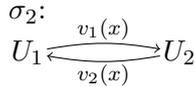
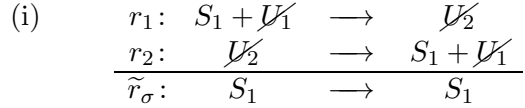
\rule{0.9\textwidth}{.4pt}

\medskip
\begin{minipage}[h]{0.22\linewidth}
\begin{tikzpicture}[inner sep=1pt]
\node (X2) at (0,-2) {$U_1$};
\node (X3) at (2,-2) {$U_2$};
\node (s) at (0,-1.5) {$\sigma_2$:};
\draw[->] (X2) to[out=10,in=170] node[above,sloped] {\scriptsize $v_1(x)$} (X3);
\draw[->] (X3) to[out=190,in=-10] node[below,sloped] {\scriptsize $v_2(x)$} (X2);
\end{tikzpicture}
\end{minipage} \qquad
\begin{minipage}[h]{0.69\linewidth}
\begin{enumerate}[(i)]
\item 
  \qquad $ 
\begin{array}[t]{cccc}
r_1\colon& S_1+\cancel{U_1} &\ce{->} & \cancel{U_2}\\
r_2\colon& \cancel{U_2} &\ce{->} & S_1+ \cancel{U_1}\\
\hline
\widetilde{r}_\sigma\colon & S_1 &\ce{->} &S_1
\end{array}
$
\item $y_{\widetilde{r}_\sigma} = S_1= S_1 = y'_{\widetilde{r}_\sigma}$
(see also Remark~\ref{trivialreact}). $\widetilde{r}_\sigma$ is not added to $\mmR_2$.
\end{enumerate}
\end{minipage} 

\medskip
\rule{0.9\textwidth}{.4pt}

\medskip
\begin{minipage}[h]{0.22\linewidth}
\begin{tikzpicture}[inner sep=1pt]
\node (X2) at (1,-2) {$U_1$};
\node (s) at (0,-1.5) {$\sigma_3$:};
\draw[->] (X2) edge [loop above] node[above,sloped] {\scriptsize $v_4(x)$} (X2);
\end{tikzpicture}
\end{minipage}\qquad 
\begin{minipage}[h]{0.69\linewidth}
\begin{enumerate}[(i)]
\item  \qquad
$
\begin{array}[t]{cccc}
r_4\colon& S_2+\cancel{U_1} &\ce{->} & S_3+\cancel{U_1}\\
\hline
\widetilde{r}_\sigma\colon & S_2 &\ce{->} &S_3
\end{array}
$
\item $y_{\widetilde{r}_\sigma} = S_2 \neq S_3 = y'_{\widetilde{r}_\sigma}$.
\item  
\hspace{-0.4cm}
\begin{tabular}{l}
 \begin{tikzpicture}[inner sep=1pt]
 \node (Gamma) at (-1.2,2) {$\Gamma(\sigma_3)=$};
 \draw[snake=brace] (-0.4,1.2) -- (-0.4,2.8);
  \draw[snake=brace] (4.9,2.8) -- (4.9,1.2);
\node (X2) at (0,2) {$U_1$};
\node (X3) at (2,2) {$U_2$,};
\draw[->] (X3) to[out=190,in=-10] node[below,sloped] {\scriptsize $v_2(x)$} (X2);
\draw[->] (X2) edge [loop above] node[above,sloped] {\scriptsize $v_4(x)$} (X2); 
\node (X21) at (2.6,2) {$U_1$};
\node (X31) at (4.6,2) {$U_2$};
\draw[->] (X31) to[out=220,in=-40] node[below,sloped] {\scriptsize $v_3(x)$} (X21);
\draw[->] (X21) edge [loop above] node[above,sloped] {\scriptsize $v_4(x)$} (X21);
\node (X31) at (5.3,2) {$\neq \emptyset$};
\end{tikzpicture}
\end{tabular}
The reaction $S_2\ce{->} S_3$ belongs to $\wR_2$.
\item Rate function: $q(x) (v_2(x)+v_3(x))v_4(x)$
\end{enumerate}
\end{minipage}

\medskip
\rule{0.9\textwidth}{.4pt}
\end{enumerate}

Hence, the reduced reaction network consists of two reactions:
\[
S_1 \ce{->[q(x)v_1(x)v_3(x)]} S_2 \ce{->[q(x)(v_2(x)+v_3(x))v_4(x)]} S_3.
\]
\end{example}

\begin{remark}\label{simplifspecies}
If $\mmU$ is noninteracting, each reaction involves at most one species in $\mmU$ in the reactant  and product. 
Each node $U_i$ of a cycle $\sigma$ in $\mmG_\mmU$ is the source (respectively, target) of exactly one edge of $\sigma$. Therefore,
   in step (\ref{cancelincycle}), each species in $\mmU$ that is a node of a cycle appears exactly once as reactant and once as product.  
\end{remark}

\begin{remark}\label{simpl}
We have defined a reaction network as a multidigraph to simplify the presentation, in particular in connection with Theorem \ref{rednet}, and hence also allowed multiple reactions between the same complexes.
The reduced reaction network 
 might, however, be further simplified by collapsing multiple reactions into one reaction without changing the ODE system, that is, using $\widehat{\mmG}$ as explained in Remark \ref{rk:hat}. The rate function of a  collapsed  reaction is  the sum of the rate functions being joined. See Example~\ref{2c1rcont}.
\end{remark}

\begin{example}\label{2c1rcont}
We continue with Example \ref{2c1r}. The two components of the graph $\mmG_\mmU$ are strongly connected, thus $\mmU$ is linearly eliminable.  The set $\Delta$ consists of two cycles: $\sigma_1$, with the edges with labels $k_1x_S$ and $k_3$, and $\sigma_2$, with the edges with labels $k_4x_S$ and $k_6$.
Both cycles give rise to the reaction $S\ce{->}S_p$ and the rate 
functions are respectively
$$\wk_1(x)=q_1(x) k_1k_3x_S\quad\textrm{ and }\quad\wk_2(x)=q_2(x) k_4k_6x_S,$$
for $q_1(x)=T_1 \Big/(k_1x_S+k_2+k_3)$ and $q_2(x)=T_2 \Big/(k_4x_S+k_5+k_6)$.
By Remark \ref{simpl}, we might further simplify the network to
$
S\ce{->[\wk_1(x)+\wk_2(x)]}S_p.
$
\end{example}

\begin{example}\label{cyclenoreact}
Consider the following reaction network:
\begin{align*}
S_1+U_1  \ce{->[\kappa_1\hspace{-1pt}(x,u)]} S_2+U_2 \quad
S_3+U_2  \ce{->[\kappa_2\hspace{-1pt}(x,u)]} S_4+U_1 \quad 
 S_4+U_2  \ce{->[\kappa_3\hspace{-1pt}(x,u)]} S_3+U_3.
\end{align*}
For the noninteracting set $\mmU=\{U_1,U_2,U_3\}$ and a $\mmU$-linear kinetics, the multidigraph $\mmG_{\mmU}$  
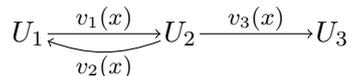
\begin{center}
\begin{tikzpicture}[inner sep=1pt]
\node (X1) at (0,0) {$U_1$};
\node (X4) at (2,0) {$U_2$};
\node (X8) at (4,0) {$U_3$};
\draw[->] (X1) to[out=0,in=180] node[above,sloped] {\footnotesize $v_1(x)$} (X4);
\draw[->] (X4) to[out=200,in=340] node[below,sloped] {\footnotesize $v_2(x)$} (X1);
\draw[->] (X4) to[out=0,in=180] node[above,sloped] {\footnotesize $v_3(x)$} (X8);
\end{tikzpicture}
\end{center}
has a spanning tree rooted at $U_3$. Thus $\mmU$ is linearly eliminable. 
Consider the unique cycle $\sigma$ of $\mmG_{\mmU}$. The graph $\sigma\setminus e$ with $e\colon U_1\rightarrow U_2$ 
 cannot be extended to a spanning tree rooted at $s(e)=U_1$. Therefore, $\Gamma(\sigma)=\emptyset$ and the reduced reaction network obtained by elimination of $\mmU$ has no reactions.
Since the graph $\mmG_{\mmU}$ does not have a spanning tree rooted at a node other than $U_3$,  we have that $\varphi(x)=(0,0,T)$, for the conservation law $u_1+u_2+u_3=T$. 
Since the kinetics is $\mmU$-linear and $U_3$ is not involved in any reactant, it follows that $\kappa(x,\varphi(x))=(0,0,0)$. Thus the ODE system of the reduced reaction network \eqref{eq:newODE} is $\dot{x}=0$.
\end{example}

\section{Basic properties of the reduced reaction network}\label{BasicProp}

In this section we study some basic properties of the reduced reaction network in Definition \ref{defrednet}. In particular, Section \ref{kinetics} is concerned with the kinetics of the reduced reaction network. Section \ref{conservation laws} is concerned with conservation laws. 

In what follows, a reaction network $(\mmC,\mmR)$ on $\mmS$, 
a kinetics $\kappa$ and a linearly eliminable set $\mmU\subseteq \mmS$ are given. 
We let $(\widetilde{\mmC},\widetilde{\mmR})$ 
be the reduced reaction network on $\mmU^c$ obtained by elimination of $\mmU$.

\subsection{Rate functions and kinetics}\label{kinetics}

In this section 
we prove that if the original rate functions vanish whenever one of the concentrations of a reactant is zero, then so do the rate functions of the reduced reaction network (Proposition \ref{prop:pisigmastandard}). A refined result is obtained when the kinetics of the original reaction network is  mass-action  (Proposition \ref{kineticsmassaction}). 
The proofs of Propositions~\ref{prop:pisigmastandard}, \ref{prop:special} and \ref{kineticsmassaction} are given in Subsection~\ref{sec:proofsBasicProp}.

 As noted in Remark~\ref{rk:kineticszero}, a biochemically meaningful rate function fulfils $\kappa_r(x)=0$ whenever  the concentration of one of the reactant species is zero, that is,
$\supp(y_r) \not\subseteq \supp(x)$ implies $\kappa_r(x)=0. $
Based on this, we say that a function $f$ defined on a subset $V\subseteq \R^n$ is  \textbf{standard} for $r\in \mmR$ if 
\begin{equation}\label{eq:standard}
\text{for all} \quad x\in V\colon \quad \supp(y_r) \not\subseteq \supp(x) \quad  \Rightarrow\quad  f(x)=0. 
\end{equation}
We say that the function is \textbf{fully standard} if the reverse implication in \eqref{eq:standard} holds as well .
The rate function $\kappa_r$ is  (fully) standard if it is a  (fully)  standard function for $r$. 
 Similarly,  the kinetics $\kappa$ is (fully) standard if the rate function of all $r\in \mmR$ is (fully) standard. 
 
The next proposition states that if $\kappa(x,u)$ is standard, then  so is   $\widetilde{\kappa}(x)$.

\begin{proposition}[Standard kinetics]\label{prop:pisigmastandard} 
\begin{enumerate}[(i)] 
\item Let $\sigma\in \Delta$ and consider the reaction $\widetilde{r}_\sigma\in \widetilde{\mmR}_2$.
If the kinetics $\kappa(x,u)$ is standard (resp. fully standard) on $\Omega\times\R_{\ge 0}^m$,  then  $\pi(\sigma)$  is
standard (resp. fully standard) on $\widetilde{\Omega}$ for $\widetilde{r}_\sigma$.  
\item If the kinetics $\kappa(x,u)$ is standard, then so is $\widetilde{\kappa}(x)$.
\end{enumerate}
\end{proposition}

The kinetics of the reduced reaction network might not be fully standard, even if the kinetics of the original network is fully standard.
For a reaction $\widetilde{r}\in \widetilde{\mmR}_1$ defined by $r\in \mmR\setminus \mmR_\mmU$, the rate function $\widetilde{\kappa}_{\widetilde{r}}(x)$ is fully standard
provided $\kappa_r(x,u)$ is. 
However, other things might occur for reactions $\widetilde{r}_\sigma\in \widetilde{\mmR}_2$, where  the rate function might    vanish if
  $$\sum\limits_{\gamma\in \Gamma(\sigma)} \pi(\gamma\setminus\sigma)=0\quad \text{ for some}\quad x\in \widetilde{\Omega},$$
 cf. equation~\eqref{defPi}.  
We characterise when this happens in Proposition~\ref{prop:special} below.
Let $\widetilde{r}_{\sigma}$ be a reaction and $x\in \R^p_{\geq 0}$ be 
positive in the coordinates corresponding to the reactants in $\widetilde{r}_{\sigma}$.
We show below that the rate function of $\widetilde{r}_{\sigma}$ vanishes at $x$, if and only if, the steady state concentrations of the species in the cycle $\sigma$ are zero whenever the concentrations of the species in $\mmU^c$ are given by  $x$. 

\begin{proposition}[Fully standard kinetics]\label{prop:special}
Assume that $\kappa(x,u)$ is fully standard on $\Omega\times\R_{\ge 0}^m$.
Let $\sigma\in \Delta$  
and $x\in \widetilde{\Omega}$ be such that  
$\supp(y_{\widetilde{r}_{\sigma}})\subseteq \supp(x)$. Then the following two statements are equivalent:
\begin{enumerate}[(i)]
\item $\Pi(\sigma)$ vanishes at $x$. 
\item $\varphi_j(x)=0$, for all $j\in \{1,\dots,m\}$ such that  $U_j\in \mmU$ is a node in $\sigma$.
\end{enumerate} 
\end{proposition}

It follows from the proposition   that the rate function of $\widetilde{r}_\sigma$ vanishes if and only if $\k_{r(e)}(x,\varphi(x))=0$ for all $e\in\sigma$ with $s(e)=U_j\in \mmU$. Indeed, the kinetics is $\mmU$-linear and hence $u_j$ is a factor of the rate function of the reaction $r(e)$.

If the vectors $x\in\widetilde{\Omega}$ for which Proposition~\ref{prop:special} applies are precisely those that satisfy $x_i=0$ for a certain $i\in\{1,\dots,p\}$,  then the reaction $\widetilde{r}_{\sigma}$ can be modified by adding $S_i$ to its reactant and product, such that $\widetilde{\k}_{\widetilde{r}_{\sigma}}$ becomes fully standard for the modified reaction. 
This modification does not alter the ODE system. We do not apply this modification here.

If the original  reaction network is endowed with mass-action kinetics, then we have a sharper result  than 
Proposition~\ref{prop:pisigmastandard}.
In particular,   the function $\pi(\sigma)$ has the form of a mass-action rate function for $\widetilde{r}_\sigma$.
Whether or not the rate function  $\widetilde{\kappa}_{\widetilde{r}_\sigma}(x)$ of the reaction is of mass-action type depends on the specific form of $q_{\mmH(\sigma)}(x)$ and $\sum\limits_{\gamma\in \Gamma(\sigma)} \pi(\gamma\setminus\sigma)$.

\begin{proposition}[Mass-action kinetics]\label{kineticsmassaction}
Assume that the reaction network $(\mmC,\mmR)$ is endowed with mass-action kinetics. For $\sigma\in \Delta$,  $$\pi(\sigma)=k_{\sigma}x^{y_{\widetilde{r}_{\sigma}}},\quad x\in\widetilde{\Omega},$$
for some constant $k_{\sigma}>0$. 
Moreover, if the reactant of $\widetilde{r}_{\sigma}$ involves the species $S_i$, then the concentration $x_i$ is a factor of $\Pi(\sigma)$ with exponent at least $(y_{\widetilde{r}_{\sigma}})_i$. 
\end{proposition}

The proposition says that the exponents of the concentrations in $\Pi(\sigma)$  might be larger than their stoichiometric coefficients, for the species involved in the reactant of $\widetilde{r}_\sigma$.
There might also be factors that depend on concentrations of species that are not involved in the reactant, as illustrated in the next example.  
The example illustrates also the following phenomenon: when two identical reactions arising from different cycles in $\mmG_\mmU$ are joined and their respective rate functions are added (cf. Remark \ref{simpl}),
then the rate function of the reaction might become fully standard.

\begin{example}\label{extfact} 
Consider the  reaction network with mass-action kinetics:
\begin{equation}\label{reactextfact}
S_1+U_1  \ce{->[k_1]}  S_2+U_2  \quad U_2   \ce{->[k_2]}  U_1  \quad 
S_3+U_3  \ce{->[k_3]}   S_1+U_2  \quad U_2  \ce{->[k_4]}  U_3 
\end{equation}
and the noninteracting set $\mmU=\{U_1,U_2,U_3\}$. The multidigraph $\mmG_{\mmU}$ is 
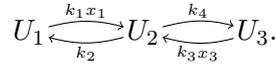
\begin{center}
\begin{tikzpicture}[inner sep=1pt]
\node (S1) at (1.5,0) {$U_2$};
\node (S2) at (0,0) {$U_1$};
\node (S3) at (3,0) {$U_3.$};
\draw[->] (S2) to[out=15,in=165] node[above,sloped] {\footnotesize $\scriptstyle k_1x_1$} (S1);
\draw[->] (S1) to[out=195,in=345] node[below,sloped] {\footnotesize $\scriptstyle k_2$} (S2);
\draw[->] (S1) to[out=15,in=165] node[above,sloped] {\footnotesize $\scriptstyle k_4$} (S3);
\draw[->] (S3) to[out=195,in=345] node[below,sloped] {\footnotesize $\scriptstyle k_3x_3$} (S1);
\end{tikzpicture}
\end{center}
The reduced reaction network obtained by elimination of $\mmU$ is
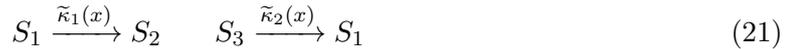
\begin{equation}\label{redreactextfact}
 S_1\ce{->[\wk_1(x)]}S_2 \hspace{20pt}  
S_3\ce{->[\wk_2(x)]}S_1
 \end{equation}
where $\wk(x)=(\wk_1(x),\wk_2(x))$ is
\begin{align*}
\wk(x) & = q(x) k_1 k_3 \Big( k_2 x_1\boldsymbol{x_3},    k_4\boldsymbol{x_1} x_3\Big),  & 
q(x) & =\frac{T}{k_2k_3x_3 +   k_1k_3x_1x_3 + k_{1}k_{4}x_1}. 
\end{align*}
Concentrations that appear with an exponent larger than the  stoichiometric coefficient  of the corresponding species are marked in bold.

We have $\widetilde{\Omega}=\R^3_{\geq 0} \setminus \{x \st x_1=x_3=0\}$.
The rate function of the first reaction in \eqref{redreactextfact}  vanishes for $x_3=0$, even though  $S_3$ is not in the reactant. This reaction corresponds to the left-cycle of $\mmG_{\mmU}$, which involves $U_1,U_2$.  
By Proposition~\ref{prop:special}, the functions $\varphi_1(x),\varphi_2(x)$ must vanish  when $x_3=0$. Indeed:
\begin{align*}
\varphi(x)=& q(x) \, \big( k_2k_3\boldsymbol{x_3}, \, k_1k_3x_1\boldsymbol{x_3}\big).
\end{align*}

Consider now the reaction network \eqref{reactextfact} with an additional reaction:
\[
S_1+U_3\ce{->[k_5]}S_2+U_3.
\]
The set  $\mmU=\{U_1,U_2,U_3\}$ is also noninteracting. The new reaction defines a self-edge  at $U_3$ in $\mmG_{\mmU}$ with label $k_5x_1$, and hence the reactions in \eqref{redreactextfact}  are also reactions of the new reduced reaction network with the same kinetics.
The cycle given by the self-edge belongs to  $\Delta$ and gives rise to an additional reaction $S_1\ce{->}S_2$ with rate function $q(x)k_1k_4k_5\boldsymbol{x_1^2}$. This reaction is considered now twice with different rate functions. 
After joining the two reactions $S_1\ce{->}S_2$, the reduced reaction network is \eqref{redreactextfact} with the rate function $\wk_1(x)$ replaced by:
$$
\wk_1(x)= q(x) k_1  k_2 k_3 x_1\boldsymbol{x_3} + q(x)k_1k_4k_5\boldsymbol{x_1^2} =q(x)k_1 \boldsymbol{(k_2k_3x_3+k_4k_5x_1)}x_1.
$$
 With this transformation, the kinetics of  \eqref{redreactextfact} is fully standard in $\widetilde{\Omega}$.  
\end{example}

\subsection{Conservation laws of the two networks}\label{conservation laws}

Let $\widetilde{S}\subseteq \R^p$ be the stoichiometric subspace of the reduced reaction network. 
In this section we compare the orthogonal complement  of the stoichiometric subspaces, $S^{\bot}$ and $\widetilde{S}^{\bot}$, which define the conservation laws of the reaction networks.  We show in Theorem \ref{relCL} that the projection of $S^{\bot}$ on $\mmU^c$ is contained in $\widetilde{S}^{\bot}$, and that the two spaces agree if the connected components of  $\mmG_{\mmU}$ are strongly connected. 
For example if $x_1+x_2+x_3=T$ is a conservation law of the original reaction network and  $S_3$ is eliminated, then $x_1+x_2=\widetilde{T}$ is a conservation law of the reduced reaction network.

Let $W$ be the vector subspace of $S^\perp$ generated by the vectors $\omega^\mmH$ defined in 
 \eqref{eq:conscut} for each connected component $\mmG_\mmH$ of $\mmG_\mmU$ such that $*\notin \mmN_\mmH$. Any vector in $S^\perp$ with support in $\mmU$ belongs to $W$. That is, $W$ is the kernel of the map $\zeta\colon S^\perp \rightarrow \R^p$ (the restriction of  the projection $\zeta$ in \eqref{projection} to $S^\perp$).  The proof of the next theorem is given in Subsection~\ref{sec:proofsBasicProp}.
\begin{theorem}[Conservation laws]\label{relCL}
\begin{enumerate}[(i)] 
\item 
$\zeta(S^{\bot})  \subseteq \widetilde{S}^{\bot}$.
\label{CL1}
\item If every connected component of $\mmG_{\mmU}$ is strongly connected, then $\zeta(S^{\bot})= \widetilde{S}^{\bot}$.
\label{CL2}
\end{enumerate}
\end{theorem}

Theorem~\ref{relCL} gives that $\zeta$ induces an injective map from the quotient vector space $S^{\bot}/ W$ to $\widetilde{S}^{\bot}$, which is an isomorphism when all connected components of $\mmG_{\mmU}$ are strongly connected.
In this case, a minimal set of conservation laws of the reduced reaction network is simply obtained from a minimal set of conservation laws of the original reaction network by ``deleting''  $u_i$, $i=1,\dots,m$, from each conservation law, and any conservation law only in $u$.

\begin{example}
By deleting $u_1=x_4$ and $u_2=x_5$ from the  first equation in the minimal set of conservation laws \eqref{CLrunex} in Example \ref{runex:2}, we obtain the minimal set of conservation laws of the reduced reaction network in Example~\ref{ex:runningcycles}  that consists of
\[
x_1+x_2+x_3=\widetilde{T}.
\] 
\end{example}

\begin{example}
We consider a mechanism that consists of  two substrates $A,B$ that are converted into two products $P,Q$ through a series of reactions catalysed by an enzyme $E$  \cite[Chapter 5]{enz-kinetics}. It is an example of a \emph{bi-bi} model in the notation of Cleland \cite{FH07}:
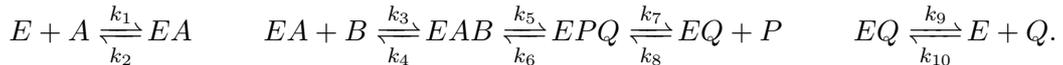
\begin{align*}
E +A&  \ce{<=>[k_1][k_2]} \! EA  & EA+B &\ce{<=>[k_3][k_4]} \! EAB  \ce{<=>[k_5][k_6]} \! EPQ \ce{<=>[k_7][k_8]} \! EQ+P & EQ &\ce{<=>[k_9][k_{10}]} \! E+Q. 
\end{align*}
We let $U_1=E$, $U_2=EA$, $U_3=EAB$, $U_4=EPQ$ and $E_5=EQ$.
We assume mass-action kinetics.
A minimal set of conservation laws  is
\begin{align*}
T_E &=u_1+u_2+u_3+u_4+u_5 & T_{A+Q}&=x_{A}+x_{Q}+u_2+u_3+u_4+u_5\\
T_{B+Q}&=x_{B}+x_{Q}+u_3+u_4+u_5 &T_{B+P}&=x_{B}+x_{P}+u_4+u_5.
\end{align*}
We consider the reduced reaction network obtained by elimination of  the set 
$\mmU=\{U_1,U_2,U_3,U_4,U_5\}$.
The multidigraph $\mmG_{\mmU}$ is strongly connected
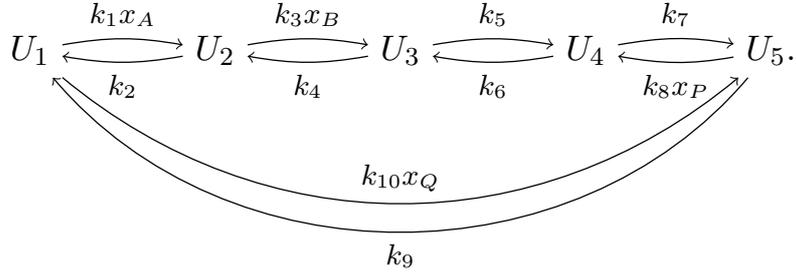
\begin{center}
\resizebox {0.7\linewidth}{!}{
\begin{tikzpicture}
\node (E) at (0,-2) {$U_1$};%{$E$};
\node (EA) at (2,-2) {$U_2$};%{$EA$};
\node (EAB) at (4,-2){$U_3$}; %{$EAB$};
\node (EPQ) at (6,-2){$U_4$}; %{$EPQ$};
\node (EQ) at (8,-2){$U_5$.}; %{$EQ$};
\draw[->] (EA) to[out=10,in=170] node[above,sloped] {\footnotesize $k_3x_B$} (EAB);
\draw[->] (EAB) to[out=190,in=-10] node[below,sloped] {\footnotesize $k_4$} (EA);
\draw[->] (EAB) to[out=10,in=170] node[above,sloped] {\footnotesize $k_5$} (EPQ);
\draw[->] (EPQ) to[out=190,in=-10] node[below,sloped] {\footnotesize $k_6$} (EAB);
\draw[->] (EPQ) to[out=10,in=170] node[above,sloped] {\footnotesize $k_7$} (EQ);
\draw[->] (EQ) to[out=190,in=-10] node[below,sloped] {\footnotesize $k_8x_P$} (EPQ);
\draw[->] (E) to[out=10,in=170] node[above,sloped] {\footnotesize $k_1x_A$} (EA);
\draw[->] (EA) to[out=190,in=-10] node[below,sloped] {\footnotesize $k_2$} (E);
\draw[->] (EQ) to[out=230,in=310] node[below,sloped] {\footnotesize $k_9$} (E);
\draw[->] (E) to[out=320,in=220] node[above,sloped] {\footnotesize $k_{10}x_Q$} (EQ);
\end{tikzpicture}
}
\end{center}
By Theorem~\ref{relCL}, a minimal set of conservation laws for the reduced reaction network is
\begin{align*}
x_{A}+x_{Q}&=\widetilde{T}_{A+Q},&x_{B}+x_{Q}&=\widetilde{T}_{B+Q},&x_{B}+x_{P}&=\widetilde{T}_{B+P}.
\end{align*}

Let us find the reduced reaction network. 
$\Delta$ contains two cycles: the cycle with the edges with labels $k_1x_A$, $k_3x_B$, $k_5$, $k_7$ and $k_9$  that meets all nodes clockwise, and the cycle with the edges with the rest of the labels that meets all nodes anti-clockwise.
For both cycles, $\Gamma(\sigma)=\{\sigma\}$ (Remark \ref{rmk:compGammasigma}) and the reduced reaction network is:
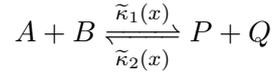
\[
A+B \ce{<=>[\wk_1(x)][\wk_2(x)]} P+Q
\]
where $\wk_1(x)=q(x)k_1k_3k_5k_7k_{9}x_Ax_B$, $\wk_2(x)=q(x) k_2k_4k_6k_8k_{10}x_Px_Q$
and $q(x)$ equals $T_E$ divided by
\begin{align*}
&  (k_{4}k_{6}+k_{4}k_{7}+k_{5}k_{7})(k_{2}k_9+k_1k_9x_{A} +k_2k_{10}x_{Q}) 
+k_{4}k_{6}k_{8}x_{P} ( k_{{1}}x_{A}+k_{2}) \\ &  +k_{3}k_{5}k_{7}x_B ( k_{9}+k_{10}x_{Q}) +( k_{7}(k_{5}+k_{9}) +(k_{5}+k_{6})(k_9+k_8 x_P))k_1k_3x_Ax_B\\
&+( k_{4}(k_{2} +k_{{6}})+(k_{5}+k_{6})(k_2+k_3x_{B}))k_8k_{10}x_{P}x_{Q}.
\end{align*}
In this case $\widetilde{\Omega}=\R^4_{\geq 0}$. 
The reduced reaction networks for this example and for the \emph{ping-pong bi-bi} mechanism considered in the introduction  have the same reactions. The two networks differ only in the  factor $q(x)$.
\end{example}

If some connected components of $\mmG_\mmU$ are not strongly connected, then the cases  $\widetilde{S}^{\bot}=\zeta(S^{\bot})$ and  $\widetilde{S}^{\bot}\neq\zeta(S^{\bot})$ are both possible. 

\begin{example}
Consider the reaction network in Example \ref{cyclenoreact}. A minimal set of conservation laws consists of
\[
u_1+u_2+u_3=T_1, \quad  x_1+x_2=T_2, \quad  x_3+x_4=T_3,\quad  x_1-x_3-u_1+u_3=T_4.\] 
By Theorem~\ref{relCL},  the following are conservation laws of the reduced reaction network:
\[
x_1+x_2=\widetilde{T}_2,\qquad x_3+x_4=\widetilde{T}_3,\qquad  x_1-x_3=\widetilde{T}_4.
\]
However, the reduced reaction network has no reactions and therefore $\widetilde{S}^{\bot}=\R^4$. Here, the multidigraph $\mmG_{\mmU}$ is not strongly connected and $\zeta(S^{\bot})\subsetneq \widetilde{S}^{\bot}$.
\end{example}

\begin{example}\label{noCLcorr}
Consider the reaction network
\[
S_1+U_1\ce{->[\kappa_1(x,u)]}S_2+U_2\hspace{25pt} S_2\ce{->[\kappa_2(x,u)]}S_1
\]
and the set $\mmU=\{U_1,U_2\}$. A minimal set of conservation laws of the reaction network consists of
$u_1+u_2=T_1$ and $x_1+x_2=T_2$.
The multidigraph $\mmG_{\mmU}$ is
\begin{center}
\begin{tikzpicture}[inner sep=1pt]
\node (X1) at (0,-2) {$U_1$};
\node (X3) at (2,-2) {$U_2$};
\draw[->] (X1) to[out=0,in=180] node[above,sloped] {\footnotesize $v_1(x)$} (X3);
\end{tikzpicture}
\end{center}
which is acyclic. The reduced reaction network has one reaction $S_2\ce{->}S_1$, and a minimal set of conservation laws consists of $x_1+x_2=T_2$. The multidigraph $\mmG_{\mmU}$ is not strongly connected and $\widetilde{S}^{\bot}= \zeta(S^{\bot})$.
\end{example}

\section{Iterative elimination}\label{severalsteps}

In this section we discuss  stepwise elimination and show that in certain cases one might obtain different reduced reaction networks, depending on the order by which species are eliminated. This naturally has practical implications. Biologically, it relates to the situation in which species at different time-scales are eliminated from the network. The fastest species are removed first, then the second fastest and so on. As the resulting network might differ from the network obtained by removing the species all at once, a proper partitioning of species into time-scales might be warranted. 
  
  The setting is the following: Let 
 $$\emptyset=\mmU_0\subset\mmU_1\subset \mmU_2\subset \dots \subset \mmU_l=\mmU$$ 
 be an increasing sequence of linearly eliminable sets. We compare the situation where $\mmU$ is eliminated at once (direct elimination) to the situation where the sets $\mmU_i\setminus \mmU_{i-1}$,  $i=1,\dots,l$, are  eliminated iteratively (iterative elimination).

If each set $\mmU_i\setminus \mmU_{i-1}$, $i=1,\dots, l$, is the node set of  a connected component of $\mmG_{\mmU}$ (or a union of such sets), then direct and iterative elimination yield the same result by construction: both elimination of variables and the reduced reaction network consider the multidigraph $\mmG_{\mmU}$ componentwise.

Another scenario in which iterative and direct elimination coincide is as follows.
Consider the connected  component $\mmG_\mmH$ of $\mmG_\mmU$ that contains $*$ (if any) and 
let $\mmG_{\mmH}^*$ be the multidigraph obtained by removal of the node $*$  and all the edges that have $*$ as  source or target. This multidigraph might now have several connected components. 
The next proposition (proven in Subsection~\ref{sec:proofssteps}) states that elimination of the species in $\mmH$
or iterative elimination of the species in the connected components of $\mmG_{\mmH}^*$ yield the same reduced reaction network.
\begin{proposition}[Splitting the connected component with node $*$]\label{nonconect}
Assume that $\mmG_\mmU$ is connected and contains $*$.
Let $\mmU=\mmH_1\sqcup \mmH_2$ be a decomposition of $\mmU$  into two disjoint subsets such that  $\mmG_\mmU^*=\mmG_{\mmH_1}^* \sqcup \mmG_{\mmH_2}^*$. Then the reduced reaction network obtained by elimination of $\mmU$ agrees with the reduced reaction network obtained by  elimination  of $\mmH_1$ followed by elimination of $\mmH_2$.
\end{proposition}

See Example~\ref{exinterm} in Subsection~\ref{intermediates} for an illustration.

\medskip

In general, iterative elimination might not yield the same reaction network as direct elimination.  We present  some examples  below.

\begin{example}
Consider the reaction network with mass-action kinetics
\[
S_1\cee{->[k_1]}U_1+S_2\cee{->[k_2]} U_2 \cee{->[k_3]} U_1+S_3\cee{->[k_4]}S_4,
\]
and the linearly eliminable sets
$\mmU_1=\{U_1\}\subset \mmU_2=\{U_1,U_2\}$ with multidigraphs
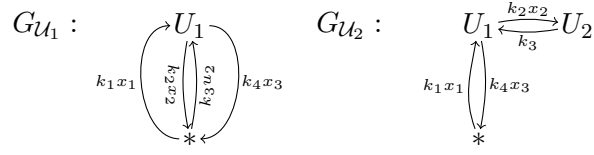
\begin{center}
\begin{tikzpicture}[inner sep=1pt]
\node (Y1) at (-0.1,0) {$U_1$};
\node (*1) at (-0.1,-1.5) {$*$};
\node (Y2) at (3.7,0) {$U_1$};
\node (Z) at (5,0) {$U_2$};
\node (*2) at (3.7,-1.5) {$*$};
\node (G1) at (-2,0) {$G_{\mathcal{U}_1}:$};
\node (G2) at (2,0) {$G_{\mathcal{U}_2}:$};
\draw[->] (*1)  to[out=180,in=180] node[left] {\footnotesize $\scriptstyle k_1x_1$} (Y1);
\draw[->] (Y1) to[out=260,in=100] node[below,sloped] {\footnotesize $\scriptstyle k_2x_2$} (*1);
\draw[->] (*1) to[out=80,in=280] node[below,sloped] {\footnotesize $\scriptstyle k_3u_2$} (Y1);
\draw[->] (Y1) to[out=360,in=0] node[right] {\footnotesize $\scriptstyle k_4x_3$} (*1);
\draw[->] (Y2) to[out=10,in=170] node[above,sloped] {\footnotesize $\scriptstyle k_2x_2$} (Z);
\draw[->] (Z) to[out=190,in=-10] node[below,sloped] {\footnotesize $\scriptstyle k_3$} (Y2);
\draw[->] (Y2) to[out=280,in=80] node[right] {\footnotesize $\scriptstyle k_4x_3$} (*2);
\draw[->] (*2) to[out=110,in=260] node[left] {\footnotesize $\scriptstyle k_1x_1$} (Y2);
\end{tikzpicture}
\end{center}
The four cycles of $\mmG_{\mmU_1}$ give that the reduced reaction network obtained by elimination of $\mmU_1$ is
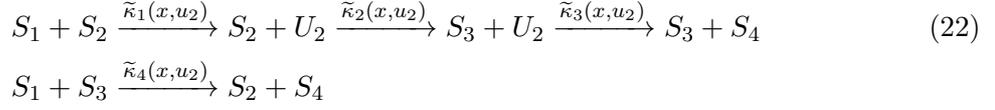
\begin{align}
S_1+S_2&\cee{->[\wk_1(x,u_2)]}S_2+ U_2 \cee{->[\wk_2(x,u_2)]}S_3+U_2 \cee{->[\wk_3(x,u_2)]} S_3+ S_4 \label{eq:iterative1}\\
S_1+S_3&\cee{->[\wk_4(x,u_2)]}S_2+S_4 \nonumber
\end{align}
with kinetics
\begin{align*}
\wk(x,u_2)& =\left(\dfrac{k_1k_2x_1x_2}{k_2x_2+k_4x_3},\ \dfrac{k_2k_3x_2u_2}{k_2x_2+k_4x_3},\ \dfrac{k_3k_4x_3u_2}{k_2x_2+k_4x_3},\ \dfrac{k_1k_4x_1x_3}{k_2x_2+k_4x_3}\right). 
\end{align*}
Here $\widetilde{\Omega}=\R^5_{\geq 0} \setminus \{x_2=x_3=0\}$. 
The set $\mmU_2\setminus \mmU_1=\{U_2\}$ is linearly eliminable for the reaction network \eqref{eq:iterative1}.
Its elimination yields the reduced reaction network:
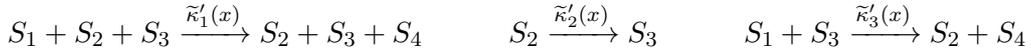
\begin{align*}
S_1+S_2+S_3&\cee{->[\wk'_1(x)]} S_2+S_3+S_4&S_2&\cee{->[\wk'_2(x)]}S_3 & 
S_1+S_3&\cee{->[\wk'_3(x)]}S_2+S_4
\end{align*}
with rate functions 
\begin{align*}
\wk'_1(x)&=\dfrac{k_1k_2k_3k_4x_1x_2x_3/(k_2x_2+k_4x_3)^2}{k_3k_4x_3/(k_2x_2+k_4x_3)}=\dfrac{k_1k_2x_1x_2}{k_2x_2+k_4x_3}, \\
\wk'_2(x)&=\dfrac{k_2k_3x_2/(k_2x_2+k_4x_3)}{k_3k_4x_3/(k_2x_2+k_4x_3)}=\dfrac{k_2x_2}{k_4x_3}\\
\wk'_3(x)&=\dfrac{k_1k_4x_1x_3}{k_2x_2+k_4x_3}=\wk_4(x,\widetilde{\varphi}_{u_2}(x)).
\end{align*}
In this case $\widetilde{\Omega}'=\{x\in \R^4_{\geq 0}  \st x_3\neq 0\}$. Since $k_2x_2+k_4x_3\neq 0$ for all $x\in \widetilde{\Omega}'$ we might simplify $k_2x_2+k_4x_3$ from  $\wk'_1(x)$ and $\wk'_2(x)$, as shown above.

The reduced reaction network obtained by direct elimination of $\mmU_2$ is:
\[
S_2\cee{->[\overline{\k}_1(x)]}S_3\qquad S_1+S_3\cee{->[\overline{\k}_2(x)]}S_2+S_4
\]
with rate functions 
\[
\overline{\k}_1(x)=\dfrac{k_2k_3x_2}{k_3k_4x_3}=\dfrac{k_2x_2}{k_4x_3},\qquad \overline{\k}_2(x)=\dfrac{k_1k_4x_1x_3}{k_3k_4x_3}=\dfrac{k_1x_1}{k_3},
\]
which is clearly different from the reduced reaction network obtained by first eliminating $\mmU_1$ and then $\mmU_2\setminus \mmU_1$.
\end{example}

When $\mmG_{\mmU}$ is connected and contains $*$, elimination of $u$ only involves the equation $\dot{u}=0$ and no conservation laws.  Therefore,  iterative elimination  is equivalent to solving a system of linear equations iteratively. It follows that the function $\varphi$ in \eqref{eq:varphi} does not depend on the chosen procedure and thus neither does the ODE system \eqref{eq:newODE}. However, as shown in the above example, the reduced reaction network may depend on the procedure, even though their associated ODE systems agree.

The next example provides an example where direct and iterative elimination result in the same reduced reaction network, and an example where the reactions but not the kinetics of the two reduced reaction networks agree.

\begin{example}\label{iterative}
Consider the  reaction network with mass-action kinetics:
\[
S_1+U_1\ce{->[k_1]}U_2 \ce{->[k_2]} U_3 \ce{->[k_3]} S_2+U_1,
\]
and the sets of noninteracting species:
\[
\mmU_1=\{U_1\}\subset \mmU_2=\{U_1,U_2\}\subset \mmU_3=\{U_1,U_2,U_3\}.
\]
The multidigraphs $\mmG_{\mmU_i}$ are  
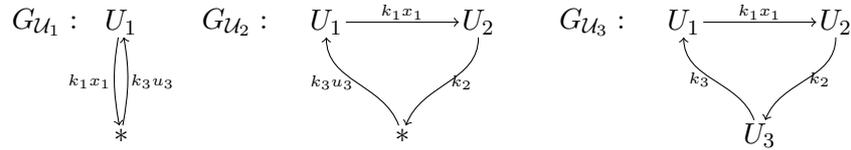
\begin{center}
\begin{tikzpicture}[inner sep=1pt]
\node (E1) at (0,0) {$U_1$};
\node (*1) at (0,-1.5) {$*$};
\node (E2) at (2.7,0) {$U_1$};
\node (Y2) at (4.7,0) {$U_2$};
\node (*2) at (3.7,-1.5) {$*$};
\node (E3) at (7.4,0) {$U_1$};
\node (Y3) at (9.4,0) {$U_2$};
\node (Z3) at (8.4,-1.5) {$U_3$};
\node (G1) at (-1,0) {$G_{\mathcal{U}_1}:$};
\node (G2) at (1.5,0) {$G_{\mathcal{U}_2}:$};
\node (G3) at (6.2,0) {$G_{\mathcal{U}_3}:$};
\draw[->] (E1) to[out=260,in=100] node[left] {\footnotesize $\scriptstyle k_1x_1$} (*1);
\draw[->] (*1) to[out=80,in=280] node[right] {\footnotesize $\scriptstyle k_3u_3$} (E1);
\draw[->] (E2) to[out=0,in=180] node[above,sloped] {\footnotesize $\scriptstyle k_1x_1$} (Y2);
\draw[->] (Y2) to[out=270,in=70]  node[right] {\footnotesize $\scriptstyle k_2$} (*2);
\draw[->] (*2)  to[out=110,in=270] node[left] {\footnotesize $\scriptstyle k_3u_3$} (E2);
\draw[->] (E3) to[out=0,in=180] node[above,sloped] {\footnotesize $\scriptstyle k_1x_1$} (Y3);
\draw[->] (Y3) to[out=270,in=70]  node[right] {\footnotesize $\scriptstyle k_2$} (Z3);
\draw[->] (Z3)  to[out=110,in=270] node[left] {\footnotesize $\scriptstyle k_3$} (E3);
\end{tikzpicture}
\end{center}
The reduced reaction network obtained by elimination of  $\mmU_1$ is
\begin{equation}\label{eq:iterative2}
\begin{array}{ccc}
S_1+U_3 \ce{->[\wk_1(x,u_2,u_3)]} S_2 + U_2 & \hspace{15pt}  U_2  \ce{->[\wk_2(x,u_2,u_3)]} U_3
\end{array}
\end{equation}
with rate functions 
\begin{align*}
\wk_1(x,u_2, u_3)& =\dfrac{k_1k_3x_1u_3}{k_1x_1}=k_3u_3, &  \wk_2(x,u_2, u_3) & = k_2u_2.
\end{align*}
Since $\widetilde{\Omega}=  \{x\in \R^4_{\geq 0} \st x_1 \neq 0\}$, $x_1$ can be  canceled in  $\wk_1(x,u_2, u_3)$.

The set $\mmU_2\setminus \mmU_1=\{U_2\}$ is linearly eliminable for the reaction network \eqref{eq:iterative2}.
The multidigraph $\mmG_{\mmU_2\setminus \mmU_1}$ is
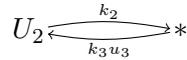
\begin{center}
\begin{tikzpicture}[inner sep=1pt]
\node (E1) at (0,0) {$U_2$};
\node (*1) at (2,0) {$*$};
\draw[->] (E1) to[out=10,in=170] node[above] {\footnotesize $\scriptstyle k_2$} (*1);
\draw[->] (*1) to[out=190,in=350] node[below,sloped] {\footnotesize $\scriptstyle k_3u_3$} (E1);
\end{tikzpicture}
\end{center}
Elimination of $\mmU_2\setminus \mmU_1$ yields the reduced reaction network
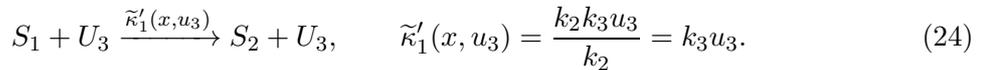
\begin{equation}\label{eq:red2}
S_1+U_3 \ce{->[\wk'_1(x,u_3)]} S_2+U_3,\qquad \wk'_1(x,u_3)=\dfrac{k_2k_3u_3}{k_2}=k_3u_3.
\end{equation}
This is also the reduced reaction network obtained by elimination of   $\mmU_2$ directly, as it can easily be seen by considering  $\mmG_{\mmU_2}$. Therefore, in this case both eliminations yield the same result.

We compare now direct elimination of $\mmU_3$ with elimination of $\mmU_3\setminus \mmU_2 = \{U_3\}$ 
 from \eqref{eq:red2}.
Both approaches provide the reduced reaction network
$$
S_1\ce{->[\wk''_1(x)]} S_2.
$$
However, the kinetics differ:
$\wk''_1(x) = k_3 T$, for iterative elimination, using the conservation law $u_3=T$, and 
$$\wk''_1(x)=  \frac{k_1k_2k_3 T x_1 }{k_1k_2x_1 + k_2x_3 + k_1k_3x_1}, \qquad \text{using }\quad u_1+u_2+u_3=T, $$
for direct elimination. 
Therefore, the kinetics of the reduced reaction network depends on whether the elimination is performed iteratively or not.
\end{example}

In general, when conservation laws are involved in the elimination procedure, then the kinetics obtained after direct and iterative elimination differ, even though the reactions of the two reduced reaction networks might be the same. The reason is that 
 the reduced ODE systems  $\dot{x} = \zeta(x,\varphi(x))$ in \eqref{eq:newODE} obtained by iterative or direct elimination are different. 
 
 To understand why, let $\mmU=\{U_1,\dots,U_{m}\}$ be a linearly eliminable set such that $\mmG_{\mmU}$ is strongly connected and does not contain the node $*$.  
 Then $u_1+\dots+u_m$ is conserved.
The set $\mmU_1=\{U_1,\dots,U_{m-1}\}$ is linearly eliminable and $\mmG_{\mmU_1}$ contains the node $*$.  Theorem \ref{elim} thus
 guarantees  that the system of equations $\dot{u}_1=\dots=\dot{u}_{m_1}=0$ in $u_1,\dots,u_{m-1}$ has a unique solution  $\varphi^1(x, u_m)$.
The reduced reaction network obtained by elimination of $\mmU_1$ has, by Theorem \ref{relCL}, the conservation law 
\begin{equation*}\label{eq:nosubsCL}
u_m=\widetilde{T}.
\end{equation*}
This implies that in the ODE system  $\dot{x} = \zeta(x,\varphi(x))$ obtained by iterative elimination of the set $\mmU_1$ followed by $\{U_m\}$, $\varphi_m(x)$ is constant equal to $T$, and the first $m-1$ components of $\varphi$ agree with $\varphi^1(x,\widetilde{T})$.

If $\mmU$ is directly eliminated, then $\varphi(x)$ is found by solving the system of equations
\begin{align*}
\dot{u}_1=\dots=\dot{u}_{m-1}&=0, & 
\sum\limits_{i=1}^mu_i&=T.
\end{align*}
The first $m-1$ equations give $u_i=\varphi^1_i(x,u_m)$ for $i=1,\dots,m-1$, and the last equation determines $u_m$ as the solution to
\begin{equation*}\label{eq:subsCL}
\sum\limits_{i=1}^{m-1}\varphi^1_i(x,u_m) +u_m=T.
\end{equation*}
Unless $\sum\limits_{i=1}^{m-1}\varphi^1_i(x,u_m)$ is constant, $u_m=\varphi_m(x)$  is not constant. 
 Hence the function $\varphi(x)$ used in $\dot{x} = \zeta(x,\varphi(x))$ depends in general on the chosen approach, and as a consequence so does the kinetics of the reduced reaction network.

\section{Relation to previous work}
\label{sec:temkin2}

 Ideas similar ideas to ours  have been proposed in the literature. 
In the first part of this section we focus on the techniques proposed in \cite{Horiuti,Temkin1}; see also  \cite{Temkin-Bonchev,Temkin_book}. We will discuss similarities and differences to our approach following the exposition in \cite{Temkin-Bonchev}.
 In the second part of the section, we justify that our construction generalises the reduction of intermediates in  \cite{feliu:intermediates}.

\subsection{Horiuti-Temkin approach}\label{sec:temkin}
In \cite{Temkin-Bonchev} the authors outline some applications of graph theory to the theory of reaction networks and introduce reduction of reaction networks by graphical means.  In that work,   the reversible reactions are treated as one reaction, while we treat them as two separate reactions in the present paper. 

The authors introduce a graph called the \emph{kinetic graph} (see \cite{Temkin1}) whose nodes are the so called \emph{intermediates}. The kinetic graph coincides with the multidigraph $\mmG_{\mmU}$ introduced in Definition \ref{defgraph} (up to the treatment of reversible reactions) and the so called intermediates form a noninteracting set in our terminology. 
Their goal is to eliminate intermediates from a reaction network  and find a \emph{minimal} mechanism that allows the computation of the production rates of the remaining species. The differences between their approach and our work arise from the details in the treatment of the reactions and the rates. 

First, in \cite{Temkin-Bonchev} not only one reduced reaction network is obtained, as in our case, but an infinite number of them. Any linear combination of the original reactions that cancel the intermediates is a possible reaction in a reduced reaction network. The set of such reactions defines a vector subspace of $S$, of which a basis is chosen. Therefore, the reduced reactions are independent  vectors in the stoichiometric subspace, and their number is minimal.

Second, the conditions imposed for the computation of the rate functions are different. In \cite{Temkin-Bonchev},   the rate functions 
are found by imposing the Horiuti-Temkin equation (see \eqref{eq:HouTem}), which involves the rate functions of the original and reduced reaction networks but does not relate to the stoichiometry of the reactions. As a consequence of the Houriti-Temkin equation, the ODE system \eqref{eq:newODE} is also satisfied for the reactions and rate functions given in \cite{Temkin-Bonchev}. We discuss next in further detail the differences on this particular point. Assume for the discussion below that $\mmR=\mmR_{\mmU}$.

For $\sigma\in \Delta$,  let $\nu_\sigma$  be the vector with $(\nu_\sigma)_i=1$ if $e_i$ is an edge of the cycle and $(\nu_\sigma)_i=0$ otherwise.
Choose an order for the set of edges in $\mmG_\mmU$ and for the set of cycles $\Delta=\big\{\sigma_1,\dots,\sigma_{|\Delta|}\big\}$, and let $\widetilde{H}$ be the $|\Delta|\times \ell$ matrix whose $i$th row is $\nu_{\sigma_i}$. With this notation, the Horiuti-Temkin equation for the reduced reaction network in Definition \ref{defrednet}, considering reversible reactions as two irreversible reactions, reads
\begin{equation}\label{eq:HouTem}
\kappa(x,\varphi(x)) = \widetilde{H}^t \wk(x).
\end{equation}
Componentwise, this condition is in our notation
\[
\kappa_i(x,\varphi(x))=\sum\limits_{(\nu_{\sigma_j})_i\neq 0}\wk_j(x)= q(x)\hspace{-10pt}\sum\limits_{\sigma\in \Delta, e(r_i)\in\sigma}\hspace{-15pt}\Pi(\sigma)\qquad i=1,\dots,\ell.
\] 
This condition might not be satisfied by the rate functions in Definition \ref{defrednet}.
We show this in the case where $\mmU$ is linearly eliminable and $\mmG_\mmU$ is connected and does not contain $*$.  We let
$q(x)=q_\mmU(x)$.

Let $i\in\{1,\dots,\ell\}$ and $U_j\in\mmU$ be involved in the reactant of $r_i\in \mmR$ (it exists because $\mmG_\mmU$ is connected and does not contain $*$). Then 
\begin{align*}
\kappa_i(x,\varphi(x))&=\varphi_j(x)v_i(x)=q(x)\sum\limits_{\tau\in\Theta(U_j)}\pi(\tau)v_i(x)=(\star).
\end{align*} 
Each term $\pi(\tau)v_i(x)$ in the sum is the label of an element in $\Gamma(\sigma)$ for some cycle $\sigma$ that contains the edge $e(r_i)$. Let $\extDelta$ be the set of all cycles of $\mmG_\mmU$. Using the definition of $\Gamma(\sigma)$  we have
\begin{align*}
(\star)&=q(x)\hspace{-5pt}\sum\limits_{\sigma\in\extDelta, \Gamma(\sigma)\neq\emptyset, e(r_i)\in\sigma}\hspace{-20pt}\Pi(\sigma)\quad \geq \quad q(x)\hspace{-10pt}\sum\limits_{\sigma\in \Delta, e(r_i)\in\sigma}\hspace{-15pt}\Pi(\sigma).
\end{align*}
The inequality arises because $\sum\limits_{e\in\sigma}\zeta(y'_{r(e)}-y_{r(e)})\neq 0$ for $\sigma\in\Delta$ and thus  the sum on the left-hand side might involve more terms than the sum on the right-hand side. Hence, the Horiuti-Temkin equation is not necessarily satisfied.

Our rate functions fulfil an equation similar to the  Horiuti-Temkin equation, once  stoichiometry is introduced. To understand this, write the ODE system \eqref{eq:ODE} 
as
\begin{equation}\label{eq:matrODE}
 g(x)=A\kappa(x),\quad x\in \Omega,
 \end{equation}
 where $A=(a_{ij})\in \R^{n\times \ell}$ is the stoichiometric matrix with $a_{ij}=(y'_{r_j}-y_{r_j})_i$.  
The stoichiometric matrix of the reduced reaction network is by Definition \ref{defrednet} $A^c\widetilde{H}^t$, where 
 $A^c$ is the matrix given by the first $p$ rows of  $A$. 
By equation \eqref{eq:matrODE} and Theorem \ref{rednet}, the kinetics   in Definition \ref{defrednet} satisfies
\[
\zeta\big(A\kappa(x,\varphi(x))\big)=A^c\widetilde{H}^t\wk(x).
\]

Finally, the kinetics obtained in \cite{Temkin-Bonchev} is not necessarily standard if the kinetics of the original reaction network is, contrary to our kinetics (cf. Section \ref{kinetics}). 
This is illustrated using the main example in \cite{Temkin-Bonchev}. 

\begin{example}
Consider the  reaction network with mass-action kinetics
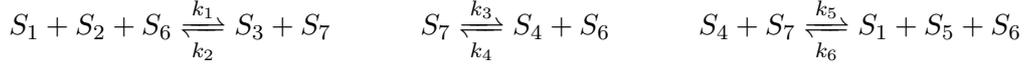
\begin{align*}
S_1+S_2+S_6 \ce{<=>[k_1][k_2]} S_3+S_7&& S_7\ce{<=>[k_3][k_4]} S_4+S_6& 
&S_4+S_7\ce{<=>[k_5][k_6]} S_1+S_5+S_6
\end{align*}
The correspondence with \cite{Temkin-Bonchev} is as follows: 
$S_1={\rm C}$, $S_2={\rm H_2O}$, $S_3={\rm H_2}$, $S_4={\rm CO}$, $S_5={\rm CO_2}$, $S_6={\rm Z_1}$ and $S_7={\rm COZ_1}$.
We consider the linearly eliminable set   $\mmU=\{S_6,S_7\}$ with  the conservation law $x_6+x_7=T$.  The reduced reaction network obtained by our procedure is
$$S_1+S_2 \ce{<=>[\wk_1(x)][\wk_2(x)]} S_3+S_4\quad
 S_1+S_2+S_4 \ce{<=>[\wk_3(x)][\wk_4(x)]} S_1+S_3+S_5 \quad
2S_4 \ce{<=>[\wk_5(x)][\wk_6(x)]} S_1+S_5$$
with kinetics
\begin{align*}
\wk(x)=&q(x) \Big( k_1 k_3 x_1 x_2,   k_2 k_4 x_3 x_4,  k_1 k_5 x_1 x_2 x_4,   k_2 k_6 x_1 x_3 x_5, k_4 k_5 x_4^2,  k_3 k_6 x_1 x_5\Big),
\end{align*}
where
$
q(x)=T \big( k_1 x_1 x_2+k_6 x_1 x_5+k_2 x_3+(k_4 +k_5) x_4+k_3\big)^{-1}.
$
The kinetics is standard for the reduced reaction network, because each rate function vanishes when one of the concentrations of the species in the reactant is zero.

One of the reduced reaction networks obtained in \cite{Temkin-Bonchev} is
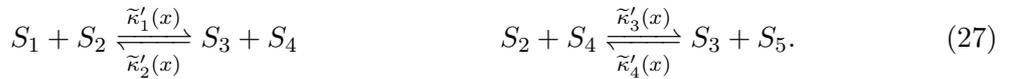
\begin{align}\label{TemkinNetwork}
&S_1+S_2 \ce{<=>[\wk'_1(x)][\wk'_2(x)]} S_3+S_4 &S_2+S_4 \ce{<=>[\wk'_3(x)][\wk'_4(x)]} S_3+S_5.
\end{align}
In \cite{Temkin-Bonchev} the rates of the reactions are considered by pairs of reversible reactions and we obtain
by their algorithm that
\begin{align*}
\wk'_1(x)-\wk'_2(x) =& q(x)( k_1 k_3 x_1 x_2- k_2 k_4 x_3 x_4+ k_3 k_6 x_1 x_5- k_4 k_5 x_4^2)\\ 
\wk'_3(x)-\wk'_4(x) =& q(x)( k_1 k_5 x_1 x_2 x_4 - k_2 k_6 x_1 x_3 x_5- k_3 k_6 x_1 x_5+ k_4 k_5 x_4^2)
\end{align*}
with $q(x)$ as above.
By collecting the terms according to their signs, we find
\begin{align*}
\wk'_1(x) =& q(x)(k_1 k_3 x_1 x_2+ k_3 k_6 x_1 x_5),& \wk'_2(x)=& q(x)( k_2 k_4 x_3 x_4+ k_4 k_5 x_4^2).
\end{align*} 
This kinetics is not standard for the reaction network \eqref{TemkinNetwork}.
\end{example}

\subsection{Intermediates}\label{intermediates}
Our construction generalises the reduction of intermediates in  \cite{feliu:intermediates}.
In \cite{feliu:intermediates}, an \emph{intermediate} is defined as a species $Y$  in a reaction network that is 
produced in at least one reaction, consumed in at least one reaction and is not involved in any  complex other than $Y$. A \emph{set of intermediates} $\mathcal{Y}$ is a subset of the species set and at the same time a subset of the set of complexes (under the identification of complexes with linear combinations of species).

Any set of intermediates is noninteracting. The multidigraph $\mmG_{\mathcal{Y}}$ contains the node $*$ and is strongly connected.  Thus, assuming mass-action kinetics, $\mathcal{Y}$ is a linearly eliminable set.
The reduced reaction network obtained by elimination of $\mathcal{Y}$ has the following properties.

Any cycle $\sigma\in  \Delta$ must contain the node $*$, because at least one reaction $r(e)$ with $e$ in $\sigma$ must involve a species in $\mathcal{Y}^c$.
Then the reactant (resp. product) of $\widetilde{r}_{\sigma}$ is  the reactant  $y_{r(e)}$ (resp. product  $y'_{r(e)}$) of the reaction corresponding to the edge $e$ with source $s(e)=*$  (resp. target  $t(e)=*$).

The label of an edge in $\mmG_{\mathcal{Y}}$ whose source is not $*$ is constant.
Thus
if $\tau$ is a spanning tree rooted at $*$, then $\pi(\tau)$ is a product of reaction rate constants $k_i$.
If $\tau$ is a spanning tree rooted at  $Y\in \mathcal{Y}$, then $\pi(\tau)$ is a product of reaction rate constants 
and the label $x^{y_{r(e)}}$ of the edge $e$ with $s(e)=*$ in the tree. 
 Thus
$$q(x)= G(k)^{-1}, \qquad\textrm{and}\qquad \Pi(\sigma)= x^{y_{\widetilde{r}_{\sigma}}} F(k),$$
where $G(k),F(k)$  are polynomials in the reaction rate constants, such that $q(x)$ is constant in $x$. 
Therefore, the reduced reaction network obtained by elimination of $\mathcal{Y}$ is a reaction network with mass-action kinetics. There is a reaction between two  complexes in $\mmC\setminus\mathcal{Y}$ in the reduced reaction network if and only if the reaction is already in the original network or there is a directed path between the two complexes through intermediates. 

\begin{example}\label{exinterm} Consider the reaction network with mass-action kinetics:
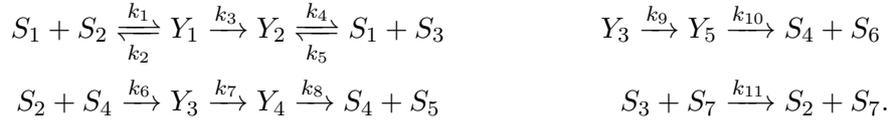
\begin{align*}
S_1+S_2  \ce{<=>[k_1][k_2]}  Y_1  &  \ce{->[k_3]}  Y_2  \ce{<=>[k_4][k_5]}  S_1+S_3  & Y_3  \ce{->[k_9]}  Y_5&   \ce{->[k_{10}]}  S_4+S_6  \\
S_2+S_4  \ce{->[k_6]}  Y_3   & \ce{->[k_7]}   Y_4   \ce{->[k_8]}   S_4+S_5  & S_3+S_7 & \ce{->[k_{11}]}  S_2+S_7.
\end{align*}
A set of intermediates of this network is $\mathcal{Y}=\{Y_1,\dots, Y_5\}$. 
The corresponding multidigraph $\mmG_{\mathcal{Y}}$ is
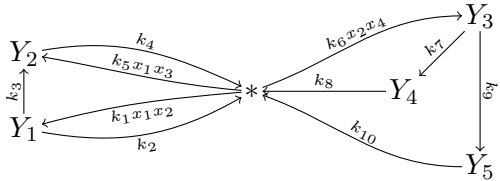
\begin{center}
\begin{tikzpicture}[inner sep=1pt]
\node (Y1) at (-3,-0.5) {$Y_1$};
\node (Y2) at (-3,0.5) {$Y_2$};
\node (Y3) at (3,1) {$Y_3$};
\node (Y4) at (2,0) {$Y_4$};
\node (Y5) at (3,-1) {$Y_5$};
\node (*) at (0,0) {$*$};
\draw[->] (*) to[out=185,in=15] node[below,sloped] {\footnotesize $\scriptstyle k_1x_1x_2$} (Y1);
\draw[->] (Y1) to[out=-10,in=210] node[below,sloped] {\footnotesize $\scriptstyle k_2$} (*);
\draw[->] (Y1) to[out=90,in=270] node[above,sloped] {\footnotesize $\scriptstyle k_3$} (Y2);
\draw[->] (Y2) to[out=10,in=155] node[above,sloped] {\footnotesize $\scriptstyle k_4$} (*);
\draw[->] (*) to[out=175,in=-10] node[above,sloped] {\footnotesize $\scriptstyle k_5x_1x_3$} (Y2);
\draw[->] (*) to[out=20,in=180] node[above,sloped] {\footnotesize $\scriptstyle k_6x_2x_4$} (Y3);
\draw[->] (Y3) to[out=225,in=45] node[above,sloped] {\footnotesize $\scriptstyle k_7$} (Y4);
\draw[->] (Y4) to[out=180,in=0] node[above,sloped] {\footnotesize $\scriptstyle k_8$} (*);
\draw[->] (Y3) to[out=270,in=90] node[above,sloped] {\footnotesize $\scriptstyle k_9$} (Y5);
\draw[->] (Y5) to[out=180,in=-20] node[above,sloped] {\footnotesize $\scriptstyle k_{10}$} (*);
\end{tikzpicture}
\end{center}
By Proposition \ref{nonconect}, the reduced reaction network obtained by elimination of  $\mathcal{Y}$ can be found by 
iteratively eliminating the sets of nodes 
 $\mathcal{Y}_1=\{Y_1,Y_2\}$ and $\mathcal{Y}_2=\{Y_3,Y_4,Y_5\}$. The multidigraphs $\mmG_{\mathcal{Y}_1}$, $\mmG_{\mathcal{Y}_2}$ correspond to the    left-subgraph and right-subgraph
 of $\mmG_{\mathcal{Y}}$, respectively.
 We 
obtain the following reduced reaction network with the specified rate functions:
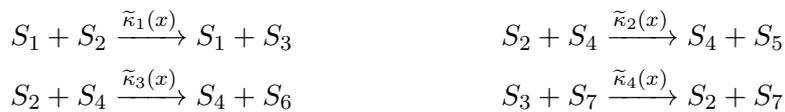
\begin{align*}
S_1+S_2 & \ce{->[\wk_1(x)]} S_1+S_3 & S_2+S_4 & \ce{->[\wk_2(x)]} S_4+S_5   \\ 
S_2+S_4 & \ce{->[\wk_3(x)]} S_4+S_6 &
S_3+S_7 & \ce{->[\wk_4(x)]} S_2+S_7 
\end{align*}
where
{\small \begin{align*}
\wk_1(x) & =\dfrac{k_1k_3k_4x_1x_2}{k_4(k_2+k_3)}=\dfrac{k_1k_3x_1x_2}{(k_2+k_3)},   &
\wk_2(x) & =\dfrac{k_6k_9k_{10}x_5x_7}{k_8k_{10}(k_7+k_9)}=\dfrac{k_6k_9x_5x_7}{k_8(k_7+k_9)},  \\
\wk_3(x) & =\dfrac{k_6k_7k_{8}x_5x_7}{k_8k_{10}(k_7+k_9)}=\dfrac{k_6k_7x_5x_7}{k_{10}(k_7+k_9)} &   \wk_4(x) & =k_{11}x_3x_7.
\end{align*}}
There is indeed a reaction between every pair of   complexes in $\mmC\setminus\mathcal{Y}$ that are connected by a directed path  through intermediates. Further, 
 the reduced reaction network has mass-action kinetics as well.
\end{example}

\section{Post-translational modification networks (PTMs)}\label{examples}
We conclude  by discussing 
reduction of PTMs. In this section we abuse notation and use $x$ for the vector of concentrations of the original  species set as well as for the concentrations in the reduced reaction network. 
 
A common feature of signalling systems is the incorporation of 
PTMs,  the attachment of some chemical group to a protein, after it has been translated. The most common example is phosphorylation. 
 A mathematical formalism to study PTM networks, that is, a network combining several PTMs, was introduced in \cite{TG-rational}. 

 In \cite{TG-rational,fwptm} PTMs are considered from the point of view of variable elimination, where so-called substrates and intermediates are eliminated.
This provides a system of equations that depends on the enzyme concentrations only. 
Here we study the reduced reaction network on the set of substrates obtained by elimination of the sets of enzymes and intermediates.

We start by giving the definition of a PTM network, which is slightly more general than the one in \cite{TG-rational,fwptm}.
The   species set of a PTM network is the disjoint union of three non-empty species sets:
\begin{itemize}
\item a set of substrates $\mathcal{S}=\{S_1,\dots,S_{p}\}$,
\item a set of enzymes $\mathcal{E}=\{E_1,\dots, E_{m_1}\}$ and
\item a set of intermediates $\mathcal{Y}=\{Y_1,\dots, Y_{m_2}\}$, in the sense of Subsection~\ref{intermediates}.
\end{itemize} 
Allowed reactions, taken with mass-action kinetics, are  of these five types: 
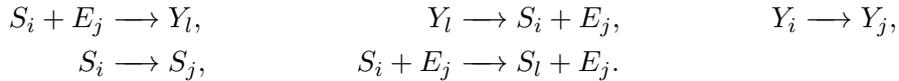
\begin{align*}
  S_i+E_j&\ce{->}Y_l,&  Y_l&\ce{->}S_i+E_j,&Y_i&\ce{->}Y_j,\\
S_i&\ce{->} S_j,&S_{i}+E_{j}&\ce{->}S_{l}+E_{j}.
  \end{align*}
The bottom types  are not considered in \cite{TG-rational,fwptm}.
We assume that  any path
\begin{equation}\label{eq:path}
S_{i_1}+E_{j_1}\ce{->} Y_{l_1}  \ce{->} \cdots \ce{->} Y_{l_t} \ce{->}S_{i_2}+E_{j_2}
\end{equation}
through intermediates
satisfies $j_1=j_2$. This provides a decomposition of $\mathcal{Y}$   into at most $m_1$ disjoint subsets according to the enzyme that ultimately produces them, or to which they dissociate.

With $\mmU=\mathcal{E}\cup \mathcal{Y}$, the multidigraph $\mmG_\mmU$ has a connected component for each enzyme,  which  is strongly connected by the hypothesis  and does not contain $*$. Thus $\mmU$ is linearly eliminable.

The next proposition (proven in Subsection \ref{sec:proofsexamples}) describes the reactions of the reduced reaction network $(\widetilde{\mmC},\widetilde{\mmR})$ obtained by elimination of $\mmU$. Recall the decomposition $\widetilde{\mmR} = \widetilde{\mmR}_1 \cup \widetilde{\mmR}_2$  given in Definition~\ref{defrednet}.
A reaction $S_i\rightarrow S_j$ in the reduced reaction network belongs to $\widetilde{\mmR}_1$ if it already belongs to $\mmR$. Therefore,  we only consider $\widetilde{\mmR}_2$ in the  proposition below. 

\begin{proposition}[PTM networks]\label{prop:ptm}
Let $(\mmC,\mmR)$ be a PTM network and $(\widetilde{\mmC},\widetilde{\mmR})$ be the reduced reaction network obtained by elimination of  $\mmU=\mmE\cup\mathcal{Y}$. Let $S_{i_1},S_{i_2}\in \mmS$ be two substrates.
The reaction $S_{i_1}\rightarrow S_{i_2}$ belongs to $\widetilde{\mmR}_2$ if and only if 
there is a path as \eqref{eq:path} from $S_{i_1}+E$ to $S_{i_2}+E$  in $(\mmC,\mmR)$ for some $E\in \mmE$.
\end{proposition}

We conclude that if two substrates do not interact with a common enzyme $E$, there is no reaction between them in $\widetilde{\mmR}_2$.

\begin{example} 
Consider  the network in Example \ref{exinterm}, which is a PTM network  with $\mathcal{Y}$ as in the example, $\mathcal{E}=\{S_1,S_4,S_7\}$ a set of enzymes and 
$\mathcal{S}=\{S_2,S_3,S_5,S_6\}$ a  set of substrates. 
By Proposition~\ref{prop:ptm}, the reactions of the reduced reaction network obtained by elimination of $\mathcal{E}\cup \mathcal{Y}$
are
\[
S_2 \ce{<=>} S_3 \hspace{25pt} S_2 \ce{->} S_5\hspace{25pt}S_2 \ce{->} S_6.
\]
\end{example}

\begin{example}[$n$-site phosphorylation system]\label{ex:nsite}
We consider an $n$-site sequential distributive phosphorylation mechanism, which consists of a substrate $S$ that contains $n$ ordered phosphorylation sites.  
We let $S_0$ denote the unphosphorylated form  and  $S_i$ denote the phosphorylated form in which sites $1$ to $i$ are phosphorylated. We assume there is a kinase $E$   that catalyses all  phosphorylation steps and, similarly, a phosphatase $F$  that catalyses dephosphorylation steps. The reaction network associated with this system is
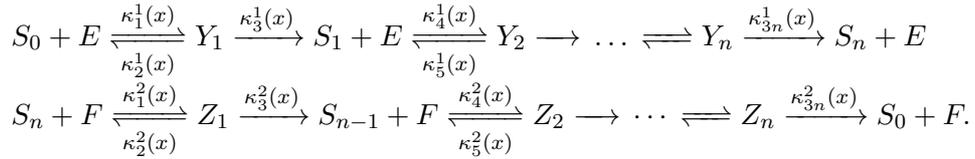
\[
\begin{array}{l}
S_0+E\ce{<=>[\kappa_1^1(x)][\kappa_2^1(x)]}Y_1 \ce{->[\kappa_3^1(x)]} S_1+E \ce{<=>[\kappa_4^1(x)][\kappa_5^1(x)]}Y_2 \ce{->} \dots \ce{<=>}Y_n \ce{->[\kappa_{3n}^1(x)]} S_n+E \\
S_n+F\ce{<=>[\kappa_{1}^2(x)][\kappa_{2}^2(x)]}Z_1 \ce{->[\kappa_{3}^2(x)]} S_{n-1}+F \ce{<=>[\kappa_{4}^2(x)][\kappa_{5}^2(x)]} Z_2\ce{->} \cdots  \ce{<=>} Z_{n} \ce{->[\kappa_{3n}^2(x)]} S_0+F.
\end{array}
\]
Let $\mmU=\{E,F,Y_1,\dots,Y_n,Z_1,\dots,Z_n\}$ and assume the kinetics is $\mmU$-linear. By Proposition~\ref{prop:ptm}, the reduced reaction network obtained by elimination of  $\mmU$ is:
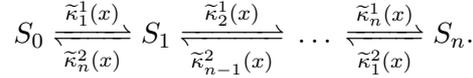
\[
S_0\ce{<=>[\wk_1^1(x)][\wk^2_n(x)]} S_1 \ce{<=>[\wk^1_2(x)][\wk^2_{n-1}(x)]} \dots \ce{<=>[\wk^1_n(x)][\wk^2_1(x)]} S_n.
\]
In order to find the kinetics, we consider the multidigraph $\mmG_{\mmU}$, which has two connected components: 
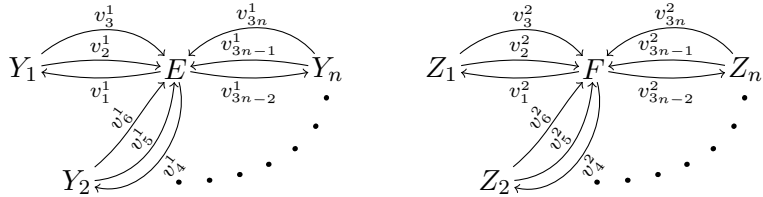
\begin{center}
\begin{tikzpicture}[inner sep=1pt, circle dotted/.style={dash pattern=on .05mm off 4mm,line cap=round}]
\node (E) at (0,0) {$E$};
\node (X1) at (-2,0) {$Y_1$};
\node (X2) at (-1.3,-1.5) {$Y_2$};
\node (X3) at (0,-1.5) {};
\node (Xn) at (2,0) {$Y_n$};
\node (F) at (5.5,0) {$F$};
\node (Y1) at (3.5,0) {$Z_1$};
\node (Y2) at (4.2,-1.5) {$Z_2$};
\node (Y3) at (5.5,-1.5) {};
\node (Yn) at (7.5,0) {$Z_n$};
\draw[->] (X1) to[out=10,in=170] node[above,sloped] {\scriptsize $ v_{\scriptscriptstyle 2}^{\scriptscriptstyle 1}$} (E);
\draw[->] (X1) to[out=40,in=130] node[above] {\scriptsize $ v_{\scriptscriptstyle 3}^{\scriptscriptstyle 1}$} (E);
\draw[->] (E) to[out=190,in=-10] node[below,sloped] {\scriptsize $ v_{\scriptscriptstyle 1}^{\scriptscriptstyle 1}$} (X1);
\draw[->] (X2) to[out=5,in=265] node[above,sloped] {\scriptsize $ v_{\scriptscriptstyle 5}^{\scriptscriptstyle 1}$} (E);
\draw[->] (X2) to[out=40,in=230] node[above,sloped] {\scriptsize $ v_{\scriptscriptstyle 6}^{\scriptscriptstyle 1}$} (E);
\draw[->] (E) to[out=285,in=-15] node[below,sloped] {\scriptsize $ v_{\scriptscriptstyle 4}^{\scriptscriptstyle 1}$} (X2);
\draw[->] (Xn) to[out=170,in=10] node[above,sloped] {\scriptsize $ v_{\scriptscriptstyle 3n-1}^{\scriptscriptstyle 1}$} (E);
\draw[->] (Xn) to[out=130,in=45] node[above] {\scriptsize $ v_{\scriptscriptstyle 3n}^{\scriptscriptstyle 1}$} (E);
\draw[->] (E) to[out=-10,in=190] node[below,sloped] {\scriptsize $ v_{\scriptscriptstyle 3n-2}^{\scriptscriptstyle 1}$} (Xn);
\draw[line width = 0.7mm, circle dotted] (X3) to[out=0,in=270] (Xn);
\draw[->] (Y1) to[out=10,in=170] node[above,sloped] {\scriptsize $ v_{\scriptscriptstyle 2}^{\scriptscriptstyle 2}$} (F);
\draw[->] (Y1) to[out=40,in=130] node[above] {\scriptsize $ v_{\scriptscriptstyle 3}^{\scriptscriptstyle 2}$} (F);
\draw[->] (F) to[out=190,in=-10] node[below,sloped] {\scriptsize $ v_{\scriptscriptstyle 1}^{\scriptscriptstyle 2}$} (Y1);
\draw[->] (Y2) to[out=5,in=265] node[above,sloped] {\scriptsize $ v_{\scriptscriptstyle 5}^{\scriptscriptstyle 2}$} (F);
\draw[->] (Y2) to[out=40,in=230] node[above,sloped] {\scriptsize $ v_{\scriptscriptstyle 6}^{\scriptscriptstyle 2}$} (F);
\draw[->] (F) to[out=285,in=-15] node[below,sloped] {\scriptsize $ v_{\scriptscriptstyle 4}^{\scriptscriptstyle 2}$} (Y2);
\draw[->] (Yn) to[out=170,in=10] node[above,sloped] {\scriptsize $ v_{\scriptscriptstyle 3n-1}^{\scriptscriptstyle 2}$} (F);
\draw[->] (Yn) to[out=130,in=45] node[above] {\scriptsize $ v_{\scriptscriptstyle 3n}^{\scriptscriptstyle 2}$} (F);
\draw[->] (F) to[out=-10,in=190] node[below,sloped] {\scriptsize $ v_{\scriptscriptstyle 3n-2}^{\scriptscriptstyle 2}$} (Yn);
\draw[line width = 0.7mm, circle dotted] (Y3) to[out=0,in=270] (Yn);
\end{tikzpicture}
\end{center}
The rate functions $v_i$ depend  on $x$.
The connected components give rise to  conservation laws with total amount $T^1$ and $T^2$, respectively.
The elements of $\Delta$  are: 
\noindent
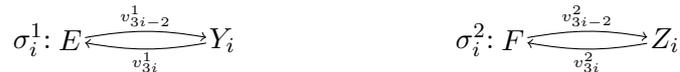
\begin{center}
\begin{minipage}[h]{0.3\linewidth}
\begin{tikzpicture}[inner sep=1pt]
\node (E) at (-2,0) {$E$};
\node (X1) at (0,0) {$Y_i$};
\node (s) at (-2.5,0) {$\sigma^1_i$:};
\draw[->] (E) to[out=10,in=170] node[above,sloped] {\tiny $v_{\scriptscriptstyle {3i-2}}^{\scriptscriptstyle 1}$} (X1);
\draw[->] (X1) to[out=190,in=-10] node[below,sloped] {\tiny $ v_{\scriptscriptstyle 3i}^{\scriptscriptstyle 1}$} (E);
\end{tikzpicture}
\end{minipage}\hspace{1cm}
\begin{minipage}[h]{0.3\linewidth}
\begin{tikzpicture}[inner sep=1pt]
\node (F) at (-2,0) {$F$};
\node (Y1) at (0,0) {$Z_i$};
\node (s) at (-2.5,0) {$\sigma^2_i$:};
\draw[->] (F) to[out=10,in=170] node[above,sloped] {\tiny $v_{\scriptscriptstyle {3i-2}}^{\scriptscriptstyle 2}$} (Y1);
\draw[->] (Y1) to[out=190,in=-10] node[below,sloped] {\tiny $ v_{\scriptscriptstyle 3i}^{\scriptscriptstyle 2}$} (F);
\end{tikzpicture}
\end{minipage}
\end{center}
for $i=1,\dots,n$.
For $l=1,2$ and $i=1,\dots,n$ the reaction with rate function $\wk_i^l(x)$ corresponds to the cycle $\sigma_i^l$ and is as follows:
\begin{align*}
\wk_i^l(x) & =q_l(x)v_{3i}^lv_{3i-2}^l\prod\limits_{j=1,j\neq i}^{n}(v_{3j-1}^l+v_{3j}^l)\\
q_l(x) & =T^l\, \Big( {\prod\limits_{j=1}^{n}(v_{3j-1}^l+v_{3j}^l)+\sum\limits_{i=1}^{n}v_{3i-2}^l\prod\limits_{j=1,j\neq i}^{n}(v_{3j-1}^l+v_{3j}^l)}\Big)^{\!-1}.
\end{align*}
\end{example}

\section{Proofs}\label{sec:proofs}

 In this section we present the technical details and proofs of the results stated in the previous sections. This section is organised such that it follows the structure of the rest of the paper. 

\subsection{Preliminaries}\label{sec:proofspreliminaries}

We give the precise definition of the digraph $\widehat{\mmG}$ introduced in Section~\ref{preliminaries} and some associated  functions that will be used later on.

\begin{definition}\label{defhat} 
Let $\mmG=(\mmN, \mmE)$ be a labeled multidigraph and let $\mmE'=\{e\in \mathcal{E}\st s(e)\neq t(e)\}$ be the set of edges that are not self-edges. The labeled digraph $\widehat{\mmG}=(\widehat{\mmN}, \widehat{\mathcal{E}})$ associated with $\mmG$ is the graph with 
\begin{align*}
\widehat{\mmN} & =\{N\in\mmN\st \exists\, e\in\mmE' \text{ with } t(e)=N \text{ or } s(e)=N\} \\
\widehat{\mathcal{E}}& =\{N_1\rightarrow N_2\st \exists\, e\in \mmE' \text{ with } s(e)=N_1\text{ and } t(e)=N_2\}.
\end{align*}
\end{definition}
Note that by  definition $\widehat{\mmN}\subseteq \mmN$ and  that the inclusion might be strict if there is a node in $\mmG$ that is only connected to itself.
We define a surjective map from $\mmE'$   to $\widehat{\mmE}$ as follows: 
\begin{equation}\label{eq:beta}
\begin{array}{cccc}
\beta\colon &\mmE' &\rightarrow  &\widehat{\mathcal{E}}\\
&e  &\mapsto & s(e)\rightarrow t(e).

\end{array} 
\end{equation}
Using this map, the labeling  $\widehat{\pi}$ for $\widehat{\mmG}$ is defined as 
\[
\widehat{\pi}(\widehat{e})=\sum\limits_{e\in \beta^{-1}(\widehat{e})} \pi(e).
\]

\subsection{Elimination of variables}\label{sec:proofselimination}

Theorem~\ref{elim} is stated and proven in \cite{Fel_elim} using the digraph $\widehat{\mmG}_{\mmU}$ instead of 
the multidigraph $\mmG_\mmU$ as we do here. This affects the definition of $q_\mmH(x)$ as well as the sets of spanning trees. Either way the functions $\varphi$  agree because  the computations performed using  $\widehat{\mmG}_{\mmU}$ or $\mmG_\mmU$ agree. It is shown in the next lemma.

\begin{lemma} \label{trees2graphs}
Let $\mmG=(\mmN, \mmE)$ be a labeled multidigraph   and let $\widehat{\mmG}=(\widehat{\mmN}, \widehat{\mmE})$ be the associated digraph given in Definition~\ref{defhat}. Let $N\in\widehat{\mmN}$  and let  $\widehat{\Theta}(N)$,  $\Theta(N)$  be the set of spanning trees rooted at $N$ of  $\widehat{\mmG}$ and $\mmG$, respectively. Then
\[
\sum_{\widehat{\tau} \in \widehat{\Theta}(N)}\widehat{\pi}(\widehat{\tau})=\sum_{\tau \in \Theta(N)}\pi(\tau). 
\] 
\end{lemma}

\begin{proof}
Since a spanning tree cannot contain a self-edge, the map  $\beta$ in \eqref{eq:beta} extends to a surjective map from $\Theta(N)$ to $\widehat{\Theta}(N)$.
 In particular, 
\begin{equation}\label{hattrees}
\Theta(N)= \coprod\limits_{\widehat{\tau} \in \widehat{\Theta}(N)}\beta^{-1}(\widehat{\tau}).
\end{equation}
Let $\widehat{\tau}$ be a  spanning tree rooted at $N$ in $\widehat{\mmG}$. Then, 
\[
\begin{array}{c}
\widehat{\pi}(\widehat{\tau})=\prod\limits_{\widehat{e}\in \widehat{\tau}}\widehat{\pi}(\widehat{e})=
\prod\limits_{\widehat{e}\in \widehat{\tau}}\ \sum\limits_{e\in\beta^{-1}(\widehat{e})} \pi(e)=
\sum\limits_{\tau\in\beta^{-1}(\widehat{\tau})}\ \prod\limits_{e\in\tau}\pi(e)=
\sum\limits_{\tau \in \beta^{-1}(\widehat{\tau})}\pi(\tau).
\end{array}
\]  
Therefore,   using \eqref{hattrees} we obtain
\[
\sum_{\tau \in \Theta(N)}\pi(\tau) = \sum_{\widehat{\tau} \in \widehat{\Theta}(N)}\ \sum\limits_{\tau\in\beta^{-1}(\widehat{\tau})} \pi(\tau) =\sum_{\widehat{\tau} \in \widehat{\Theta}(N)}\widehat{\pi}(\widehat{\tau}).
\]
\end{proof}

\subsection{Reduced Network}\label{sec:proofsrednet} 

In this section we first prove   that the definition of $\Gamma(\sigma)$ in \eqref{eq:gammasigma},
$$ \Gamma(\sigma) = \{ \tau \cup e \st \sigma\setminus e \subset \tau, \ \tau\in \Theta(s(e)) \},$$
is independent of the chosen edge $e$, cf. Proposition~\ref{constGamma}. Subsequently, we prove Theorem \ref{rednet}.

We define the following set of sub-multidigraphs of $\mmG_{\mmU}$:  
\begin{equation}\label{defGamma}
\Gamma:=\big\{ \gamma=\tau \cup e\st \tau\in\Theta(s(e)),\quad e\in \mmG_{\mmU} \big\}.
\end{equation}
Each element of $\Gamma$ is the union of 
an edge  $e$ of $\mmG_\mmU$ 
and a spanning tree (of the connected component in $\mmG_\mmU$ that contains $e$) that is rooted  at the source of the edge,  $s(e)$.
Any multidigraph $\gamma\in\Gamma$ contains a unique cycle. Indeed, for $\gamma=\tau \cup e$ as in \eqref{defGamma},
the cycle is obtained by joining $e$ and the path in $\tau$ from $t(e)$ to $s(e)$, which exists because $\tau$ is rooted at $s(e)$.
In particular, the cycle contains the edge $e$.

\begin{proposition}\label{constGamma}
Let $\sigma$ be a cycle of  $\mmG_{\mmU}$. For any edge $e\in \sigma$,
$$ \{ \tau \cup e \st \sigma\setminus e \subset \tau, \ \tau\in \Theta(s(e)) \}=\{\gamma\in \Gamma\ |\ \sigma \subset \gamma\}.$$
\end{proposition}

\begin{proof}
Without loss of generality we assume that $\mmG_\mmU$ is connected.
If $\gamma\in\Gamma$ and $\sigma \subset \gamma$, then $\gamma=\tau\cup e$ with $\tau\in\Theta(s(e))$ and $e\in\sigma$. Clearly  $\tau$ is  a spanning tree that contains $\sigma\setminus e$.
Conversely, if $\tau\in\Theta(s(e))$ is a spanning tree containing $\sigma\setminus e$, then $\tau\cup e\in\Gamma$ and $\sigma \subset \tau\cup e$.
\end{proof}

Therefore, $\Gamma(\sigma)$ is the subset of $\Gamma$ whose elements contain $\sigma$. This shows that 
the definition of $\Gamma(\sigma)$ in \eqref{eq:gammasigma} is independent of the choice of edge $e$.
Note that $\Gamma$ is the disjoint union of $\Gamma(\sigma)$ for all cycles $\sigma$.

The next proof shows  that the reduced ODE system \eqref{eq:newODE}  is the ODE system associated with the reduced reaction network  in Definition \ref{defrednet}. 

\begin{proof}[Proof of Theorem \ref{rednet}: \textbf{Reduced reaction network}]
Let 
\begin{align*}
f(x) &= \sum_{ \widetilde{r}\in \wR_1} \wk_{\wr}(x) (y'_{\widetilde{r}}-y_{\widetilde{r}})  + \sum_{\widetilde{r}\in \wR_2}
 \widetilde{\kappa}_{\widetilde{r}}(x)(y'_{\widetilde{r}}-y_{\widetilde{r}}).
  \end{align*}
 We want to prove that
 $ f(x)=\widetilde{g}(x) = \zeta(g(x,\varphi(x))). $
Observe  that 
\begin{align*} 
\widetilde{g}(x)  & = 
\sum\limits_{r\notin\mmRu}\kappa_r(x,\varphi(x))\zeta(y'_r-y_r) +\sum\limits_{r\in\mmRu}\kappa_r(x,\varphi(x)) \zeta(y'_r-y_r).
\end{align*}
By definition of $\wR_1$ and $\wk_{\wr}(x)$ for $\wr\in \wR_1$, the first summand of $f(x)$ and that of $\widetilde{g}(x)$ agree. Using the definition of $\wR_2$, all we  need to prove is that
\begin{equation}\label{eq:finalstep}
\sum_{\sigma\in \Delta }
 q_{\mmH(\sigma)}(x)\Pi(\sigma) \sum\limits_{e\in\sigma}\zeta(y'_{r(e)}-y_{r(e)})
 = \sum\limits_{r\in\mmRu}\kappa_r(x,\varphi(x))\zeta (y'_r-y_r).
\end{equation}
The sums on the right- and left-hand sides can be decomposed into sums over the connected components of $\mmG_\mmU$ that contain $\sigma$ or $e(r)$, respectively. Therefore, it is enough to show that \eqref{eq:finalstep} holds when $\mmG_\mmU$ 
is connected. We let   $q(x)=q_{\mmH(\sigma)}(x) = q_\mmU(x)$. 

Given $r\in \mmRu$,  either $s(e(r))=U_i$ for some $i$ if $\rho(y_r)\neq 0$, or $s(e(r))=*$ if $\rho(y_r)= 0$.
In the former case $U_i$ is the only species in $\mmU$ involved in the reaction  $r$. 
Since the kinetics $\kappa$ is $\mmU$-linear, we have for  $r\in \mmRu$
$$ \kappa_r(x,\varphi(x)) = \begin{cases}
\varphi_{i}(x) v_r(x)  & \textrm{if }\quad s(e(r))=U_i  \\
v_r(x) & \textrm{if } \quad s(e(r))=*.  
\end{cases}$$
By the definition of $q(x)$ in \eqref{qnoncut} and of $\varphi_i(x)$ in \eqref{eq:varphi}, we have
\[
\kappa_r(x,\varphi(x))  = q(x)\left( \sum\limits_{\tau \in \Theta(s(e(r)))}\pi(\tau)\right) v_r(x),\qquad \text{for all }r\in \mmRu. 
\]
Comparing this equality with \eqref{eq:finalstep}, the statement follows if the following  holds:
\begin{align}\label{eq:intstep}
&\sum_{r\in \mmRu}\hspace{3pt} \sum_{\tau \in \Theta(s(e(r)))}\pi(\tau) v_r(x) \zeta(y'_r-y_r)=
&\sum_{\sigma\in \Delta }\hspace{3pt}
\sum\limits_{e\in\sigma} \Pi(\sigma)   \zeta(y'_{r(e)}-y_{r(e)}).
\end{align}
We show the equality from right to left. Let $\extDelta$ be the set of cycles in the multidigraph $\mmG_{\mmU}$. Since either $\sum\limits_{e\in\sigma}\zeta(y'_{r(e)}-y_{r(e)})= 0$ 
or $\Pi(\sigma)= 0$ for $\sigma\in\extDelta\setminus  \Delta$, we obtain that 
\begin{align*}
 & \sum_{\sigma\in\Delta }
\sum\limits_{e\in\sigma} \Pi(\sigma) \zeta (y'_{r(e)}-y_{r(e)})  =  \sum_{\sigma\in\extDelta }
\sum\limits_{e\in\sigma} \Pi(\sigma)  \zeta(y'_{r(e)}-y_{r(e)})  \\ 
&\hspace{80pt}=\sum\limits_{\sigma\in\extDelta}\sum\limits_{e\in\sigma}\sum\limits_{\gamma\in\Gamma(\sigma)}\pi(\gamma\setminus e)\pi(e)\zeta(y'_{r(e)}-y_{r(e)})=(\star). 
\end{align*}
Each digraph $\gamma\setminus e$ in the sum is a spanning tree rooted at $s(e)$. 
There is a bijection between the set of triplets $(\sigma,e,\gamma)$ such that $\sigma\in \extDelta,e\in \sigma,\gamma\in \Gamma(\sigma)$ and the set of pairs $(e,\tau)$ such that $e$ is an edge of $\mmG_\mmU$ and $\tau$ is a spanning tree rooted at $s(e)$.
Using further that $\pi(e)= v_{r(e)}(x)$ and the correspondence \eqref{defre}, we  obtain: 
\begin{align*}
(\star)&=\sum\limits_{e\in \mmG_{\mmU}}\sum\limits_{\tau \in \Theta(s(e))}\pi(\tau)\pi(e)\zeta(y'_{r(e)}-y_{r(e)})\\
&=\sum\limits_{r\in \mmRu}\sum\limits_{\tau \in \Theta(s(e(r)))}\pi(\tau)v_r(x)\zeta(y'_{r}-y_{r}).
\end{align*}
This shows that \eqref{eq:intstep} holds, which concludes the proof.
 \end{proof}

\subsection{Basic properties of the reduced network}\label{sec:proofsBasicProp}

In this section we prove the results  about the kinetics and the conservation laws of the reduced reaction network in relation to the original reaction network.

\paragraph{\textbf{Kinetics. }}

\begin{proof}[Proof of Proposition \ref{prop:pisigmastandard}: \textbf{Standard kinetics}]
(i) By definition
\[
\pi(\sigma)=\prod_{e\in\sigma} v_{r(e)}(x),\qquad y_{\widetilde{r}_{\sigma}}=\sum\limits_{e\in\sigma}\zeta(y_{r(e)}).
\]
Hence, $\pi(\sigma)=0$ if and only if $v_{r(e)}(x)=0$ for some $e\in\sigma$. Further,  we 
have that $\supp(y_{\widetilde{r}_{\sigma}}) = \bigcup\limits_{e\in \sigma}\supp(\zeta(y_{r(e)}))$.  
Assume that the kinetics $\kappa$ is standard on $\Omega\times\R_{\ge 0}^m$.   For $x\in \widetilde{\Omega}$, we have
\begin{align*}
\supp(y_{\widetilde{r}_{\sigma}}) \not\subseteq \supp(x)   & \Leftrightarrow    \exists \ e\in \sigma \colon
\supp(\zeta(y_{r(e)}))\not\subseteq \supp(x) \\ &  \Rightarrow 
\exists\ e\in \sigma \colon  v_{r(e)}(x)=0  \quad \Leftrightarrow   \pi(\sigma)=0.
 \end{align*}
 This shows that $\pi(\sigma)$ is standard on $\widetilde{\Omega}$ for $\widetilde{r}_\sigma$. 
If the kinetics is fully standard, then the reverse of the second implication holds, showing that $\pi(\sigma)$ is also fully standard for $\widetilde{r}_\sigma$.

(ii) For $\widetilde{r}_{\sigma}\in\widetilde{\mmR}_2$ defined by  $\sigma\in\Delta$
we have $\widetilde{\kappa}_{\widetilde{r}_{\sigma}}(x)=q(x)\Pi(\sigma)$. By equation \eqref{defPi}, $\pi(\sigma)$ is a factor of $\Pi(\sigma)$ and hence $\widetilde{\kappa}_{\widetilde{r}_{\sigma}}$ is standard on $\widetilde{\Omega}$.

For $\widetilde{r}=\zeta(y_r)\rightarrow \zeta(y'_r) \in\widetilde{\mmR}_1$ with $r\in \mmR\setminus \mmR_\mmU$, we have $\widetilde{\kappa}_{\widetilde{r}}(x)=\kappa_r(x,\varphi(x))$ by Definition~\ref{defrednet}. 
Since $r\notin  \mmR_\mmU$, then $\supp(y_r)=\supp(\zeta(y_r))=\supp(y_{\widetilde{r}})$.
Since $\kappa_r$ is standard on $\Omega\times\R_{\ge 0}^m$, then $\kappa_r(x,u)$ vanishes if $x\in \Omega$ fulfils $\supp(y_{\widetilde{r}})\not\subseteq \supp(x)$. Since the denominators of  $q_\mmH(x)$ do not vanish for $x\in \widetilde{\Omega}$, we have $\widetilde{\kappa}_{\widetilde{r}}(x)=0$ whenever $\supp(y_{\widetilde{r}})\not\subseteq \supp(x)$ and $x\in \widetilde{\Omega}$, showing that $\widetilde{\kappa}_{\widetilde{r}}(x)$ is standard on $\widetilde{\Omega}$.
\end{proof}

\begin{proof}[Proof of Proposition \ref{prop:special}: \textbf{Fully standard kinetics}]
Throughout the proof all functions are assumed to be evaluated at $x$.
Assume that statement (ii) holds. 
Let $e\in\sigma$ be an edge such that $s(e)=U_j\in \mmU$. Then
\[
\Pi(\sigma)=\pi(e)\sum\limits_{\gamma\in\Gamma(\sigma)}\pi(\gamma\setminus e)\leq \pi(e) \sum\limits_{\tau\in\Theta(U_j)}\pi(\tau)=\pi(e)\dfrac{\varphi_j(x)}{q_{\mmH(j)}(x)}.
\] 
Since $\varphi_j(x)=0$,  $\Pi(\sigma)=0$ and statement (i) holds.

Assume now that statement (i) holds. 
Let $U_j \in \mmU$ be a node in  $\sigma$, $\tau$ a spanning tree rooted at $U_j$ and $e\in \sigma$  the edge with source $U_j$. 
We construct a new tree $\widehat{\tau}$ as follows: for every $U_l\neq U_j$ in $\sigma$, replace the only edge with source $U_l$ in the tree $\tau$ by the edge in $\sigma$ with source node $U_l$. The obtained subgraph $\widehat{\tau}$ is also a spanning tree rooted at $U_j$ that satisfies
$\widehat{\tau}\cup e\in \Gamma(\sigma)$.

By assumption (i) and the definition of $\Pi(\sigma)$, $\pi(\widehat{\tau}\cup e)=0$. Since $\kappa(x,u)$ is fully standard, $\pi(e')\neq 0$ for  all $e'\in \sigma$ by Proposition~\ref{prop:pisigmastandard}. 
Therefore there must be an edge $\widehat{e}$ of $\widehat{\tau}$ that does not belong to $\sigma$ and such that $\pi(\widehat{e})=0$.
By construction, this edge is also an edge of $\tau$ and hence $\pi(\tau)=0$. 

This proves that $\pi(\tau)=0$ for all $\tau\in \Theta(U_j)$, which implies that  $\varphi_j(x)=0$  (cf. \eqref{eq:varphi}). Hence statement (ii) holds.
\end{proof}

\begin{proof}[Proof of Proposition \ref{kineticsmassaction}: \textbf{Mass-action kinetics}]
The label of an edge of $\sigma$ is $k_{r(e)}x^{\zeta(y_{r(e)})}$, where $k_{r(e)}>0$ is the reaction rate constant of  reaction $r(e)$.  
Using the definition of  $y_{\widetilde{r}_{\sigma}}$ in Definition~\ref{defrednet}, we obtain
\[
\pi(\sigma)=\prod\limits_{e\in\sigma}k_{r(e)}x^{\zeta({y}_{r(e)})}=\left(\prod\limits_{e\in\sigma}k_{r(e)}\right) x^{\sum\limits_{e\in \sigma} \zeta({y}_{r(e)})} = 
\left(\prod\limits_{e\in\sigma}k_{r(e)}\right) x^{y_{\widetilde{r}_{\sigma}}}.
\]
Hence $\pi(\sigma)$ has the claimed form with $k_{\sigma}=\prod\limits_{e\in\sigma}k_{r(e)}$.  
The last statement is a consequence of \eqref{defPi}.
\end{proof}

\paragraph{\textbf{Conservation laws. }}
We prove here Theorem \ref{relCL}. Before that, we  introduce some graphical constructions and a technical lemma necessary for the proof. 

Let $\mmG=(\mmN, \mathcal{E})$ be a multidigraph and assume that the sets $\mmN=\{N_1,\dots,N_n\}$ and $\mmE=\{e_1,\dots,e_\ell\}$ are ordered. We define the following objects:
\begin{enumerate}
\item The \textbf{incidence matrix} $C_{\mmG}$ of $\mmG$ is the 
$n\times \ell$ real matrix such that
\[
(C_{\mmG})_{ij}=\left\{ 
\begin{array}{rl}
1 & \quad\text{if } \quad N_i=t(e_j)\neq s(e_j)\\ 
-1 & \quad\text{if }\quad N_i=s(e_j)\neq t(e_j)\\
0& \quad \text{otherwise}.
 \end{array}\right.
\]
\item The \textbf{cycle space} of $\mmG$ is the kernel of the incidence matrix.
\end{enumerate}

Let $\extDelta$ be the set of cycles of $\mmG$. 
For $\sigma\in \extDelta$, the vector $\nu_\sigma$ with $(\nu_\sigma)_i=1$ if $e_i$ is an edge of the cycle and $(\nu_\sigma)_i=0$ otherwise, belongs to the cycle space of $\mmG$. Moreover,   the 
elements $\nu_\sigma$ correspond to the irreducible elements in the cycle space with all non-zero components equal to one. That is, $\nu_\sigma$  cannot be expressed as the positive sum of two vectors in the cycle space with non-negative integer coordinates. The elements in $\extDelta$ are also called elementary cycles in the literature \cite{Graphs}. 
We choose an order for the set of cycles $\extDelta=\{\sigma_1,\dots,\sigma_{|\extDelta|}\}$, and let $H$ be the $|\extDelta|\times \ell$ matrix whose $i$th row is $\nu_{\sigma_i}$.

\begin{lemma}\label{inccycles} 
Let $\mmG$ be a strongly connected multidigraph. 
Then $\ker H = \im C_{\mmG}^t$.
\end{lemma}
\begin{proof}
By duality it is enough to show that 
$\im H^t=\ker C_{\mmG}$. Since for all $\sigma\in\extDelta$, $\nu_\sigma$ belongs to the cycle space, 
we have $C_{\mmG}H^t=0$, that is, $\im H^t\subseteq \ker  C_{\mmG}$.
Since $\mmG$ is strongly connected, then there exists a basis of $\ker  C_{\mmG}$ whose elements are of the form $\nu_\sigma$ for  $\sigma\in \extDelta$ (see \cite{Be85}). 
Hence $\im H^t\supseteq \ker  C_{\mmG}$ as desired. 
 \end{proof}

We have now all necessary tools to prove the relation between $S^{\bot}$ and $\widetilde{S}^{\bot}$.

\begin{proof}[Proof of Theorem \ref{relCL}: \textbf{Conservation laws}]
(\ref{CL1}) Let $\omega\in S^{\bot}$ and $\widetilde{r}\in\widetilde{\mmR}_1$ be a reaction corresponding to $r\in\mmR\setminus \mmR_{\mmU}$. Since  $\omega\cdot (y'_r-y_r)=0$ and $\rho(y'_r-y_r)=0$, we deduce that $\zeta(\omega)\cdot(y'_{\widetilde{r}}-y_{\widetilde{r}})=0$.

Let $\widetilde{r}_{\sigma}\in\widetilde{\mmR}_2$ for $\sigma\in \Delta$ and 
$\eta=\sum\limits_{e\in\sigma} (y'_{r(e)}-y_{r(e)})$. By Remark \ref{simplifspecies},  $\rho(\eta)=0$ and by definition
$ \zeta(\eta)=y'_{\widetilde{r}_{\sigma}}-y_{\widetilde{r}_{\sigma}}$. Thus we have
$$ \zeta(\omega)\cdot (y'_{\widetilde{r}_{\sigma}}-y_{\widetilde{r}_{\sigma}}) =
 \zeta(\omega)\cdot \zeta(\eta) + \rho(\omega)\cdot \rho(\eta) = \omega \cdot \eta =0.   $$
This proves  (\ref{CL1}).

(\ref{CL2}) Assume that each connected component of  $\mmG_{\mmU}$ is strongly connected. 
Given $\widetilde{\omega}\in \widetilde{S}^{\bot}$, we want to prove that there exist $\alpha_1,\dots, \alpha_m$ such that 
$$\omega=(\widetilde{\omega}_1,\dots, \widetilde{\omega}_{p},\alpha_1,\dots,\alpha_m)\in S^{\bot}.$$ That is, such that $\omega A=0$, where $A$ is the stoichiometric matrix of  $(\mmC,\mmR)$ (see Subsection~\ref{sec:temkin}).
Let $\overline{\ell}=|\mmRu|$. We order the set $\mmR$ in such a way that $\mmR\setminus \mmR_\mmU=\{r_{\overline{\ell}+1},\dots,r_{\ell}\}$. Then, $A$  can be written in block form as 
\[
\left(\begin{array}{cc} A_1 & A_3 \\ A_2 & 0_{m\times (\ell-\overline{\ell})} \end{array} \right) \quad\text{ with } A_1\in\R^{p\times \overline{\ell}},\ A_2\in\R^{m\times \overline{\ell}}\text{ and } A_3\in\R^{p\times (\ell-\overline{\ell})}.
\]
The columns of $A_3$ correspond to the reactions in $\wR_1$ and are thus vectors of $\widetilde{S}$ by Definition \ref{defrednet}.  Hence, $\omega \left(\begin{array}{c} A_3 \\  0_{m\times (\ell-\overline{\ell})} \end{array} \right)=0$ for any choice of $\alpha_1,\dots,\alpha_m$.
It follows that 
\begin{align}
\omega A=0 &\Leftrightarrow (\widetilde{\omega}_1,\dots, \widetilde{\omega}_{p})A_1+(\alpha_1,\dots,\alpha_m)A_2=0\nonumber\\
&\Leftrightarrow A_2^t\left( \begin{matrix} \alpha_1 \\ \vdots \\ \alpha_m \end{matrix}\right) = -\left( \begin{matrix} \widetilde{\omega}\cdot \zeta(y'_{r_1}-y_{r_1})\\ \vdots \\ \widetilde{\omega}\cdot \zeta(y'_{r_{\overline{\ell}}}-y_{r_{\overline{\ell}}}) \end{matrix}\right)=:v.\label{CLsyst1}
\end{align}
We can further reorder the species in $\mmU$ and the reactions in $\mmR_\mmU$, such that $A_2$ is a block diagonal matrix, where each block corresponds to one connected component of $\mmG_\mmU$. 
Due to this block structure,  system \eqref{CLsyst1} decomposes into subsystems given by the connected components of $\mmG_\mmU$.
Therefore, it is enough to prove the existence of $\alpha$ for a strongly connected graph $\mmG_\mmU$. Hence we assume $\mmG_\mmU$ is strongly connected.

Let $C_{\mmG}$ be the incidence matrix of $\mmG_{\mmU}$ and $H$ as defined above Lemma \ref{inccycles}. 
We prove that $Hv=0$.  
For $\sigma\in \Delta$, we have that  $\widetilde{r}_{\sigma}\in \widetilde{\mmR}$ and hence
\[
0=\widetilde{\omega}\cdot (y'_{\widetilde{r}_{\sigma}}-y_{\widetilde{r}_{\sigma}})=\sum_{e\in\sigma}\widetilde{\omega}\cdot\zeta(y'_{r(e)}-y_{r(e)}).
\]
For $\sigma\in\extDelta\setminus\Delta$, strong connectedness and Remark~\ref{rmk:compGammasigma} imply that $\Gamma(\sigma)\neq \emptyset$. Thus, by definition of $\Delta$,
\[
0=\sum_{e\in\sigma}\zeta(y'_{r(e)}-y_{r(e)}) \quad\textrm{and hence}\quad 0=\sum_{e\in\sigma}\widetilde{\omega}\cdot\zeta(y'_{r(e)}-y_{r(e)}).
\]
This shows that $v \in \ker H$ and hence by Lemma \ref{inccycles} we have $v \in \im C_{\mmG}^t$.

If $\mmG_\mmU$ does not contain $*$, then $A_2=C_{\mmG}$. Since $v \in\im C_{\mmG}^t$ we deduce that  system \eqref{CLsyst1} has a solution $(\alpha_1,\dots,\alpha_m)$.
If $\mmG_\mmU$  contains $*$, then $A_2=\widetilde{C}_{\mmG}$, with $\widetilde{C}_{\mmG}$ being the matrix obtained from $C_{\mmG}$ by removing the last row, corresponding to the node $*$. We can then rewrite the system of equations \eqref{CLsyst1} as
\begin{equation}\label{CLsyst2}
v=A_2^t\left( \begin{matrix} \alpha_1 \\ \vdots \\ \alpha_m \end{matrix}\right)=C_{\mmG}^t\left(\begin{matrix}\alpha_1\\\vdots\\\alpha_m\\0\end{matrix}\right).
\end{equation}
Since $v\in \im C_{\mmG}^t$, there exists a vector  $b=(b_1,\dots,b_{m+1})^t$ in $\R^{m+1}$ such that  $C_{\mmG}^t b = v$.
Since the column sums of $C_{\mmG}$ are all zero by definition, 
$(1,\dots,1)^t \in \ker C_{\mmG}^t$. Thus
$$ v = C_{\mmG}^t ( b - b_{m+1}(1,\dots,1)^t).  $$
Therefore, $\alpha_i=b_i-b_{m+1}$, for $i=1\dots,m$ is a solution to \eqref{CLsyst2}.
 \end{proof}

\subsection{Iterative elimination}\label{sec:proofssteps}

In this section we prove Proposition \ref{nonconect}. We start with an auxiliary lemma.

\begin{lemma}\label{nonconect2}
Let $\mmU$ be a linearly eliminable set such that  $\mmG_\mmU$ is connected and contains $*$. 
Let $\mmU=\mmH_1\sqcup \mmH_2$ be a decomposition of $\mmU$ such that  $\mmG_\mmU^*=\mmG_{\mmH_1}^* \sqcup \mmG_{\mmH_2}^*$. 
Let $\Theta_{1}(N)$ be  the set of spanning trees rooted at $N\in  \mmH_1\cup \{*\}$ in the sub-multidigraph $\mmG_{\mmH_1}$ of $\mmG_{\mmU}$. For any set of edges $W$ in $\mmG_{\mmH_1}$, it holds
\[q_\mmU(x)\sum\limits_{\tau \in \Theta(N), W\subseteq \tau}\pi(\tau)
=q_{\mmH_1}(x)\sum\limits_{\tau \in \Theta_{1}(N), W\subseteq \tau}\pi(\tau).
\]
\end{lemma}
\begin{proof}
Let $\Theta_{2}(*)$ be the set of spanning trees of $\mmG_{\mmH_2}$  rooted at $*$. 
Let $\tau\in \Theta(N)$ be a spanning tree of $\mmG_{\mmU}$ rooted at $N$. The path from any node
$N'\in \mmH_2$ to $N$ contains $*$ by hypothesis. Therefore, 
$\tau$ is the union of a spanning tree  $\tau_1$ of $\mmG_{\mmH_1}$ rooted at $N$ and a spanning tree $\tau_2$ of $\mmG_{\mmH_2}$ rooted at $*$. 
Reciprocally, the union of any pair of spanning trees $\tau_1 \in \Theta_{1}(N)$,  $\tau_2\in \Theta_{2}(*)$ is a spanning tree of $\mmG_{\mmU}$ rooted at $N$.
As subgraphs of $\mmG_{\mmU}$, $\tau_1$ and $\tau_2$ intersect at $*$.
Hence
$
\pi(\tau)=\pi(\tau_1)\pi(\tau_2)
$  
and we obtain
\[
\sum\limits_{\tau\in\Theta(N),W\subseteq \tau}\pi(\tau)=\left(\sum_{\tau_1\in\Theta_{1}(N),W\subseteq \tau_1}\pi(\tau_1)\right) \left(\sum_{\tau_2\in\Theta_{2}(*)}\pi(\tau_2)\right),
\]
where we use that $W$ is contained in $\mmG_{\mmH_1}$.
Using this computation with $W=\emptyset$ and $N=*$, we also have that 
$$q_\mmU(x)^{-1} = \left(\sum_{\tau_1\in\Theta_{1}(*)}\pi(\tau_1)\right) \left(\sum_{\tau_2\in\Theta_{2}(*)}\pi(\tau_2)\right)=
q_{\mmH_1}(x)^{-1}  \left(\sum_{\tau_2\in\Theta_{2}(*)}\pi(\tau_2)\right).$$
Using these expressions the statement of the lemma follows.
\end{proof}

Note that we can write
$$
\Pi(\sigma)= \pi(e)\sum\limits_{\gamma\in \Gamma(\sigma)} \pi(\gamma\setminus e)
$$
for any edge $e$ in $\sigma$. The sum is over all spanning trees rooted at $s(e)$ that contain the edges of $\sigma\setminus e$. Lemma~\ref{nonconect2}  guarantees that the computation of  $\varphi(x)$ and the rate functions $q_{\mmU}(x)\Pi(\sigma)$ of the reduced reaction network  is independent of whether we consider $\mmG_{\mmU}$, or $\mmG_{\mmH_1}$ and $\mmG_{\mmH_2}$ separately. Proposition~\ref{nonconect} now follows  because the sets of cycles of $\mmG_{\mmH_1}$ and $\mmG_{\mmH_2}$ are disjoint.

\subsection{Post-translational modification networks}\label{sec:proofsexamples}

\begin{proof}[Proof of Proposition \ref{prop:ptm}: \textbf{PTM networks}]
The only edges in $\mmG_{\mmU}$ whose associated reaction involves a substrate (that is, a species in $\mmU^c$) are those with $E$ as source or target. 

$\Leftarrow)$ 
Let 
$S_{i_1}+E \ce{->}Y_{j_1}\ce{->}  \dots \ce{->} Y_{j_s} \ce{->} E+S_{i_2}$, $s\geq 0,$ be a path in $(\mmC,\mmR)$.
This path defines a cycle $\sigma$ in $\mmG_{\mmU}$ with nodes $E$, $Y_{j_1},\dots, Y_{j_s}$ (it is a self-edge if $s=0$). This cycle is such that $\Gamma(\sigma)\neq \emptyset$ ($\sigma$ belongs to a strongly connected component)
and  gives rise to a reaction 
$
S_{i_1}\ce{->} S_{i_2}$  in $\widetilde{\mmR}_2$.

$\Rightarrow)$
If $S_{i_1}\ce{->} S_{i_2}\in \widetilde{\mmR}_2$, 
then there is a cycle $\sigma\in \Delta$ defining it. The cycle $\sigma$ must contain a unique enzyme $E\in\mmE$. 
It follows that one of the edges in $\sigma$ corresponds to a reaction with reactant $S_{i_1} + E$, and one of edges  corresponds to a reaction with product $S_{i_2} + E$. All other edges correspond to reactions between intermediates. 
The reactions corresponding to the edges in the cycle give the claimed path from $S_{i_1}+E$ to $S_{i_2}+E$. 
 \end{proof}

\paragraph{Acknowledgements}
We thank E. Tonello for pointing out a mistake in a previous version of the paper.
MS, EF, CW  are supported by The Lundbeck Foundation (Denmark). 
EF and CW  acknowledge funding from the Danish Research Council of Independent Research. MS has been supported by the project  MTM2012-38122-C03-02/FEDER from the Ministerio de Econom\'{\i}a y Competitividad, Spain.

\bibliographystyle{spmpsci}      % mathematics and physical sciences
%\bibliography{crnt}   % name your BibTeX data base

%%\bibliographystyle{spmpsci}      % mathematics and physical sciences


\end{document}